\documentclass[a4paper,oneside,12pt]{amsart}

\usepackage{amssymb}
\usepackage{color}
\usepackage[usenames,dvipsnames]{xcolor}
\usepackage[colorlinks=true,linkcolor=Blue,citecolor=ForestGreen]{hyperref}
\usepackage{amsfonts}
\usepackage{amsthm}
\usepackage{amsmath}
\usepackage{amscd}
\usepackage{geometry}
\usepackage{mathrsfs}

\newtheorem{prop}{Proposition}
\newtheorem{thm}{Theorem}
\newtheorem{lem}{Lemma}

\newtheorem{conj}{Conjecture}

\theoremstyle{definition}

\makeatletter
      \def\@setcopyright{}
      \def\serieslogo@{}
      \makeatother

\begin{document}

\title[The Dedekind zeta function]{Moments of the Dedekind zeta function and other non-primitive $L$-functions}
\author{Winston Heap}
\address{Department of Mathematics, University of York, York, YO10 5DD, U.K.}
\email{winstonheap@gmail.com}
\thanks{The author is supported by an Engineering and Physical Sciences Research Council doctoral grant}
\begin{abstract}We give a conjecture for the moments of the Dedekind zeta function of a Galois extension. This is achieved through the hybrid product method of Gonek, Hughes and Keating. The moments of the product over primes are evaluated using a theorem of Montgomery and Vaughan, whilst the moments of the product over zeros are conjectured using a heuristic method involving random matrix theory. The asymptotic formula of the latter is then proved for quadratic extensions in the lowest order case. We are also able to reproduce our moments conjecture in the case of quadratic extensions by using a modified version of the moments recipe of Conrey et al. Generalising our methods, we then provide a conjecture for moments of non-primitive $L$-functions, which is supported by some calculations based on Selberg's conjectures. 
\end{abstract}
\maketitle

\section{Introduction and statement of results}
Let $\mathbb{K}$ be a number field of discriminant $d_\mathbb{K}$ and let $\zeta_\mathbb{K}(s)$ be its Dedekind zeta function.
In this note we are interested in the asymptotic behaviour of the moments
\begin{equation}I_k(T)=\frac{1}{T}\int_T^{2T}\left|\zeta_\mathbb{K}\left(\frac{1}{2}+it\right)\right|^{2k}dt
\end{equation}
with $k$ real. The only known asymptotic for $I_k(T)$ was given by Motohashi \cite{mot} in the case where $\mathbb{K}$ is quadratic and $k=1$. He showed that 
\begin{equation}\label{second moment}
I_1(T)\sim \frac{6}{\pi ^2}L(1,\chi)^2\prod_{p|d_\mathbb{K}}\left(1+\frac{1}{p}\right)^{-1}\log^2 T
\end{equation}
where $\chi$ is the Kronecker character $(d_\mathbb{K}| \,\cdot\,)$.
Other results concerning the mean values of $\zeta_\mathbb{K}(s)$ can be found in \cite{bm,bm 2,fomenko,me,muller ded,sarnak 0}.

Similarly to the Riemann zeta function, it is difficult to even form conjectures on the higher asymptotics of $I_k(T)$. In the paper \cite{conrey ghosh}, Conrey and Ghosh  were able to provide a conjecture for the sixth moment of $\zeta(1/2+it)$. Later, Conrey and Gonek \cite{conrey gonek} described a method that could also give a conjecture for the eighth. Their methods involved mean values of long Dirichlet polynomials, and it seems these methods reach their limit with the eighth moment.  It is only recently that believable conjectures have been made for all values $k>-1/2$. These were first given by Keating and Snaith \cite{keating snaith} and took the form
\begin{equation}\frac{1}{T}\int_T^{2T}\left|\zeta\left(\frac{1}{2}+it\right)\right|^{2k}dt\sim \frac{a(k)g(k)}{\Gamma(k^2+1)}\log^{k^2}T
\end{equation}
where
\begin{equation}\label{zeta arith}a(k)=\prod_p\left(\left(1-\frac{1}{p}\right)^{k^2}\sum_{j\geq 0}\frac{d_k(p^j)^2}{p^j}\right)
\end{equation}
and 
\begin{equation}\frac{g(k)}{\Gamma(k^2+1)}=\frac{G(k+1)^2}{G(2k+1)}
\end{equation}
where $G$ is Barnes' $G$-function. Their main idea was to model the zeta function as a characteristic polynomial. This was motivated by the apparent similarities between the non-trivial zeros of the zeta function and eigenangles of matrices in the circular unitary ensemble. However, one drawback of their method was that the arithmetic factor had to be incorporated in an ad hoc fashion. Later, Gonek, Hughes and Keating \cite{hybrid} reproduced this conjecture in such a way that the arithmetic factor was included in a more natural way. In this paper we reproduce these results for the Dedekind zeta function. 

The method of Gonek, Hughes and Keating first involves expressing the zeta function as a partial product over primes times a partial product over the zeros. This uses a smoothed form of the explicit formula due to Bombieri and Hejhal \cite{bomb hej}. The equivalent for the Dedekind zeta function takes the following form  

\begin{thm}\label{hybrid prod thm}Let $X\geq 2$ and let $l$ be any fixed positive integer. Let $u(x)$ be a real, non-negative, smooth function with mass 1 and compact support on $[e^{1-1/X},\,e]$. Set
\[U(z)=\int_0^\infty u(x)E_1(z\log x)dx,\] 
where $E_1(z)=\int_z^\infty e^{-w}/w \, dw$. Then for $\sigma\geq 0$ and $|t|\geq 2$ we have 

\begin{equation}\label{hybrid prod}\zeta_\mathbb{K}(s)=P_\mathbb{K}(s,X)Z_\mathbb{K}(s,X)\left(1+O\left(\frac{X^{l+2}}{(|s|\log X)^l}\right)+O(X^{-\sigma}\log X)\right)
\end{equation}
where 
\begin{equation}\label{P prod} P_\mathbb{K}(s,X)=\exp\bigg(\sum_{\substack{\mathfrak{a}\subseteq\mathcal{O}_\mathbb{K}\\\mathfrak{N}(\mathfrak{a})\leq X}}\frac{\Lambda(\mathfrak{a})}{\mathfrak{N}(\mathfrak{a})^s\log\mathfrak{N}(\mathfrak{a})}\bigg)
\end{equation}
with 
\begin{equation}\label{von mangoldt}\Lambda(\mathfrak{a})=
\begin{cases}\log \mathfrak{N}(\mathfrak{p})& \text{\,\,if\,\,} \mathfrak{a}=\mathfrak{p}^m,\\0 &\text{\,\,otherwise,}
\end{cases}
\end{equation}
and
\begin{equation}\label{Z prod}
Z_\mathbb{K}(s,X)=\exp\left(-\sum_{\rho}U((s-\rho)\log X)\right),
\end{equation}
where the sum is over all non-trivial zeros of $\zeta_\mathbb{K}(s)$.
\end{thm}
 
Following a similar reasoning to that in \cite{hybrid} we can view formula (\ref{hybrid prod}) as a hybrid of a truncated Euler product and a truncated Hadamard product. We can then make the equivalent of their \emph{splitting conjecture} for the moments $I_k(T)$. This takes the form

\begin{conj}\label{splitting conj}Let $X,T\to\infty$ with $X\ll(\log T)^{2-\epsilon}$. Then for $k>-1/2$, we have
\begin{equation}I_k(T)\sim\Bigg(\frac{1}{T}\int_T^{2T}\left|P_\mathbb{K}\left(\frac{1}{2}+it,X\right)\right|^{2k}dt\Bigg)\times\Bigg(\frac{1}{T}\int_T^{2T}\left|Z_\mathbb{K}\left(\frac{1}{2}+it,X\right)\right|^{2k}dt\Bigg).
\end{equation}
\end{conj}

We plan to evaluate the moments of $P_\mathbb{K}$ by using the Montgomery-Vaughan mean value theorem \cite{mv 0}.  Due to the nature of how primes split, or rather, how they are not known to split in some cases, we restrict ourselves to Galois extensions. It may be possible to remove this restriction given milder conditions on $\mathbb{K}$. In section \ref{p moment sec} we show 
\begin{thm}\label{arithmetic moments thm} Let $\mathbb{K}$ be a Galois extension of degree $n$ with Galois group $G=\mathrm{Gal} (\mathbb{K}/\mathbb{Q})$ and for a given prime $\mathfrak{p}$ let $g_{\mathfrak{p}}$ denote the index of the decomposition group $G_\mathfrak{p}$ in $G$. Let $1/2\leq c<1$, $\epsilon>0$, $k>0$ and suppose that $X$ and $T\to \infty$ with $X\ll (\log T)^{1/(1-c+\epsilon)}$. Then 
\begin{equation}\label{arith moment form}
\frac{1}{T}\int_T^{2T}\left|P_\mathbb{K}\left(\frac{1}{2}+it,\,X\right)\right|^{2k}dt\sim a(k)\chi_\mathbb{K}^{nk^2}(e^{\gamma}\log X)^{nk^2}
\end{equation}
where $\chi_\mathbb{K}$ denotes the residue of $\zeta_\mathbb{K}(s)$ at $s=1$ and 
\begin{equation}\label{arith factor}a(k)=\prod_{\mathfrak{p}\subseteq\mathcal{O}_\mathbb{K}}\bigg(\left(1-\frac{1}{\mathfrak{N}(\mathfrak{p})}\right)^{nk^2}
\left(\sum_{m\geq 0}\frac{d_{g_\mathfrak{p}k}(\mathfrak{p}^m)^2}{\mathfrak{N}(\mathfrak{p})^m}\right)^{1/g_\mathfrak{p}}\bigg)
\end{equation}
with  $d_k(\mathfrak{p}^m)=d_k(p^m)=\Gamma(m+k)/(m!\Gamma(k))$. 
\end{thm}

In considering the moments of $Z_\mathbb{K}$ for Galois extensions we first express $\zeta_\mathbb{K}(s)$ as a product of Artin $L$-functions. For each individual $L$-function we then follow the heuristic argument given in section 4 of \cite{hybrid}. This essentially allows us to write the moments of $Z_\mathbb{K}$ as an expectation over the unitary group. We then assume a certain quality of independence between the Artin $L$-functions, namely, that the matrices associated to the zeros of $L(s,\chi,\mathbb{K}/\mathbb{Q})$ at height $T$, act independently for distinct $\chi$. This allows for a factorisation of the expectation and we are led to 

\begin{conj}\label{z moment conj}Let $\mathbb{K}$ be a Galois extension of degree $n$. Suppose that $X,T\to\infty$ with $X\ll (\log T)^{2-\epsilon}$. Then for $k>-1/2$ we have
\begin{multline}
\frac{1}{T}\int_T^{2T}\left|Z_\mathbb{K}\left(\frac{1}{2}+it,X\right)\right|^{2k}dt\\\sim (e^\gamma\log X)^{-nk^2} \prod_\chi
\frac{G(\chi(1)k+1)^2}{G(2\chi(1)k+1)}\Big(\log \big(q(\chi)T^{d_\chi}\big)\Big)^{\chi(1)^2k^2}
\end{multline}
where the product is over the irreducible characters of $\mathrm{Gal}(\mathbb{K}/\mathbb{Q})$, $G$ is the Barnes $G$-function,  
 $q(\chi)$ is the conductor of $L(s,\chi,\mathbb{K}/\mathbb{Q})$ and $d_\chi$ is its dimension.
\end{conj}

\noindent We remark that the dimension of an $L$-function is defined in \cite{mur}, and for Artin $L$-functions is simply the number of Gamma functions appearing in its completed form.

By combining this with Theorem \ref{arithmetic moments thm} and Conjecture \ref{splitting conj} we see that the factors of $e^\gamma\log X$ cancel, as expected, and we acquire a full conjecture for the moments of $\zeta_\mathbb{K}(1/2+it)$ when $\mathbb{K}$ is Galois. 
We note that after using $\sum_\chi \chi(1)^2=|\mathrm{Gal}(\mathbb{K}/\mathbb{Q})|$$=n$ the resulting expression in this conjecture is $\sim c\log^{nk^2} T$ for some constant $c$. Now, in the paper \cite{con farm}, Conrey and Farmer express the idea that the mean square of $\zeta(s)^k$ should be a multiple of the sum $\sum_{n\leq T}d_k(n)^2n^{-1}$, and that this multiple is the measure of how many Dirichlet polynomials are needed to capture the full moment. Their reasoning is based on a combination of the Montgomery-Vaughan mean value Theorem and the form of the sixth and eighth moment conjectures given in \cite{conrey gonek}. Assuming this idea applies to other $L$-functions, we note a result of Chandrasekharan and Narasimhan \cite{chandra nara}. They showed that for a Galois extension of degree $n$,
\begin{equation}\sum_{m\leq T}f_\mathbb{K}(m)^2\sim cT\log^{n-1} T, 
\end{equation}
where $f_\mathbb{K}(m)$ is the number of integral ideals of norm $m$ and $c$ is some constant. Applying partial summation we thus gain a result which supports our conjecture, at least in the case $k=1$ (we note the results of \cite{chandra nara} should easily extend to general $k$, and remain consistent with our conjecture). Alternatively, one could view our conjecture as adding support to the idea of Conrey and Farmer.

In this paper a particular emphasis is placed on quadratic extensions, so let us first fix our notation. We note that  
if $d_\mathbb{K}$ is the discriminant of a quadratic field and $\chi(n)=(d_\mathbb{K}|n)$ where $(\,\cdot\,|\,\cdot\,)$ is the Kronecker character, then $\chi$ is a real Dirichlet character mod 
\begin{equation}\label{q}q=\begin{cases}4|d_\mathbb{K}| &\text{if}\,\,d_\mathbb{K}\equiv 2 (\!\!\!\!\!\!\mod 4), \\ |d_\mathbb{K}|&\text{otherwise}\end{cases}\end{equation}
and $\zeta_\mathbb{K}(s)=\zeta(s)L(s,\chi)$. In section \ref{z moment sec} we prove Conjecture \ref{z moment conj} in the lowest order case. That is, we prove

\begin{thm}\label{z moment thm}Let $\mathbb{K}$ be a quadratic extension. Suppose that $X,T\to\infty$ with $X\ll (\log T)^{2-\epsilon}$. Then
\begin{equation}\frac{1}{T}\int_T^{2T}\left|Z_\mathbb{K}\left(\frac{1}{2}+it,X\right)\right|^{2}dt\sim \frac{\log T\cdot\log qT}{(e^\gamma\log X)^{2}}.
\end{equation}
\end{thm}
\noindent By combining this with Theorem \ref{arithmetic moments thm} and then comparing with Motohashi's result (\ref{second moment}), we see that Conjecture \ref{splitting conj} is true for $k=1$ in the case of quadratic extensions. 

Recently, an alternative method for conjecturing moments of primitive $L$-functions was given by Conrey et al. in \cite{cfkrs}. This comes in the form of a recipe. By using a result of the author \cite{me}, we add a modification to this recipe which allows for non-primitive $L$-functions.  In section \ref{recipe sec}, we use this modified recipe to reproduce the full moments conjecture for quadratic extensions. This is given by      

\begin{conj}\label{quad conj}Let $\mathbb{K}$ be a quadratic extension and let $a(k)$ be given by (\ref{arith factor}). Then
\begin{equation}I_k(T)\sim a(k)L(1,\chi)^{2k^2}\left(\frac{G(k+1)^2}{G(2k+1)}\right)^2\left(\log T\cdot\log qT\right)^{k^2}. 
\end{equation}
\end{conj}

Finally, in section \ref{non-prim sec} we attempt to generalise the main ideas of this paper to non-primitive $L$-functions. We restrict ourselves to reasonable $L$-functions, which is to say, we consider functions of the form 
\begin{equation}L(s)=\sum \alpha_L(n)n^{-s}=\prod_{j=1}^mL_j(s)^{e_j}
\end{equation}
where $e_j\in\mathbb{N}$ and the $L_j(s)$ are distinct, primitive members of the Selberg class $\mathcal{S}$. 
We assume that we have the functional equation
\begin{equation}\Lambda_{L_j}(s):=\gamma_{L_j}(s)L_j(s)=\epsilon_j\overline{\Lambda}_{L_j}(1-s)
\end{equation}
\noindent where $\epsilon_j$ is some number of absolute value 1 and 
\begin{equation}\gamma_{L_j}(s)=Q_j^{s/2}\prod_{i=1}^{d_j}\Gamma(s/2+\mu_{i,j})
\end{equation}
with the $\{\mu_{i,j}\}$ stable under complex conjugation. We also require that the `convolution' $L$-functions
\begin{equation}M_j(s)=\sum_{n=1}^\infty \frac{|\alpha_{L_j}(n)|^2}{n^s}
\end{equation}
 behave reasonably, in particular, that they have an analytic continuation. We then claim

\newpage
\begin{conj}\label{non-prim conj}With the notation as above, let $\alpha_{L,k}(n)$ be the Dirichlet coefficients of $L(s)^k$. Then for $k>-1/2$,
\begin{equation}\label{non-prim eq}\frac{1}{T}\int_0^T\left|L\left(\frac{1}{2}+it\right)\right|^{2k}dt\sim a_L(k)\prod_{j=1}^m\frac{G^2(e_jk+1)}{G(2e_jk+1)}\left(\log \big(Q_jT^{d_j}\big)\right)^{(e_jk)^2}
\end{equation}
where 
\begin{equation}a_L(k)=\prod_p \left(1-\frac{1}{p}\right)^{n_Lk^2}\sum_{n=0}^\infty \frac{|\alpha_{L,k}(p^n)|^2}{p^n}
\end{equation}
with $n_L=\sum_{j=1}^m e^2_j$.
\end{conj}
We remark that if $L(s)=\zeta_\mathbb{K}(s)$ with $\mathbb{K}$ Galois and we have a factorisation in terms of Dirichlet series, then the residue term $\chi_\mathbb{K}^{nk^2}$ of (\ref{arith moment form}) is a factor of $a_L(k)$.  

Note that the right hand side of (\ref{non-prim eq}) is $\sim (a_L(k)g_L(k)/\Gamma(n_Lk^2+1))\log^{n_Lk^2} T$ where
\begin{equation}g_L(k)=\Gamma(n_Lk^2+1)\prod_{j=1}^m\frac{G^2(e_jk+1)}{G(2e_jk+1)}d_j^{(e_jk)^2}.
\end{equation}
 As previously noted, one expects the mean square of $L(1/2+it)^k$ to be asymptotic to a multiple of the sum $\sum_{n\leq T}|\alpha_{L,k}(n)|^2n^{-1}$. On the assumption of Selberg's conjectures, we give an argument showing that 
\begin{equation}\sum_{n\leq T}\frac{|\alpha_{L,k}(n)|^2}{n}\sim \frac{a_L(k)}{(n_Lk^2)!}\log^{n_Lk^2} T,
\end{equation}
which adds further support to our conjecture. We also note that for integral $k$, 
\begin{equation}g_L(k)=\binom{n_Lk^2}{(e_1k)^2,\ldots,(e_mk)^2}\prod_{j=1}^m g(e_jk) d_j^{(e_jk)^2}
\end{equation}
where the first factor is the multinomial coefficient and the function $g$ is defined by $g(n)/n^2!$$=G(n+1)^2/G(2n+1)$. It is shown in \cite{con farm} that $g(n)$ is an integer, and hence $g_L(k)$ is an integer for integral $k$.

\subsection*{Acknowledgments} 

I'd like to thank Caroline Turnage-Butterbaugh  and Chris Hughes for their useful comments and suggestions.

\section{The hybrid product}\label{hybrid sec}
In this section we prove Theorem \ref{hybrid prod thm}. For this we require a  smoothed version of the explicit formula which is given in Lemma \ref{explicit lem}. The proof of this follows similarly to that of the classical explicit formula and uses the following two Lemmas. We omit their proofs since they are easily adapted from the results of \cite{mv} by using well known properties of the Dedekind zeta function. 

\begin{lem}\label{N_1 lem}Suppose $t\neq\gamma$ for any zero $\rho=\beta+i\gamma$, $0\leq\beta\leq 1$, of $\zeta_\mathbb{K}(s)$. Then
\[\sum_{\rho}\frac{1}{1+(t-\gamma)^2}\ll \log T.\]
This implies $N_1(t):=|\{\gamma :|t-\gamma|<1\}|\ll \log t$.
\end{lem}

\begin{lem}\label{log der lem}For $-1\leq\sigma\leq 2$ and $t\neq\gamma$ for any zero $\rho$ we have 
\[\frac{\zeta_\mathbb{K}^\prime(s)}{\zeta_\mathbb{K}(s)}=\sum_{\substack{\rho\\|t-\gamma|<1}}\frac{1}{s-\rho}+O(\log t)\]
\end{lem}

\begin{lem}\label{explicit lem}Let $u(x)$ be a real, nonnegative smooth function with compact support in $[1,e]$, and let $u$ be normalized so that if 
\begin{equation}v(t)=\int_t^\infty u(x)dx,\end{equation}
then $v(0)=1$.  Let 
\begin{equation}\hat{u}(z)=\int_0^\infty u(x)x^{z-1}dx\end{equation}
be the Mellin transform of $u$. Then for $s$ not a zero or a pole of $\zeta_\mathbb{K}(s)$ we have 

\begin{equation}\begin{split}\label{log der}-\frac{\zeta^\prime_\mathbb{K}(s)}{\zeta_\mathbb{K}(s)}=&\sum_{\mathfrak{a}\subseteq\mathcal{O}_\mathbb{K}}\frac{\Lambda(\mathfrak{a})}{\mathfrak{N}(\mathfrak{a})^s}v(e^{\log\mathfrak{N}(\mathfrak{a})/\log X})-\sum_{\rho}\frac{\hat{u}(1-(s-\rho)\log X)}{s-\rho}\\
&-(r_1+r_2)\sum_{m=1}^\infty\frac{\hat{u}(1-(s+2m)\log X)}{s+2m}\\ 
& -r_2\sum_{j=0}^\infty\frac{\hat{u}(1-(s+2j+1)\log X)}{s+2j+1}
-\chi_\mathbb{K}\frac{\hat{u}(1-(s-1)\log X)}{s-1}
\end{split}\end{equation} 
where $\Lambda(\mathfrak{a})$ is as in (\ref{von mangoldt}) and $r_1$, $r_2$ are, respectively, the number of real and complex embeddings $\mathbb{K}\to\mathbb{C}$.  
\end{lem}

\begin{proof}Let $c=\max\{2,2-\Re(s)\}$. By absolute convergence we have 

\begin{equation}\begin{split}\frac{1}{2\pi i}\int_{c-i\infty}^{c+i\infty}\frac{\zeta^\prime_\mathbb{K}(s+z)}{\zeta_\mathbb{K}(s+z)}\hat{u}(1+z\log X)\frac{dz}{z}=&-\sum_{\mathfrak{a}\subseteq\mathcal{O}_\mathbb{K}}\frac{\Lambda(\mathfrak{a})}{\mathfrak{N}(\mathfrak{a})^s}\frac{1}{2\pi i}\int_{c-i\infty}^{c+i\infty}\frac{\hat{u}(1+z\log X)}{\mathfrak{N}(\mathfrak{a})^z}\frac{dz}{z}\nonumber
\\=&-\sum_{\mathfrak{a}\subseteq\mathcal{O}_\mathbb{K}}\frac{\Lambda(\mathfrak{a})}{\mathfrak{N}(\mathfrak{a})^s}v(e^{\log\mathfrak{N}(\mathfrak{a})/\log X}).
\end{split}\end{equation}

Let $M_T(d)$ denote the rectangular contour with vertices $(c-iT,c+iT,-d+iT,-d-iT),\,\,d>0$. Then, by the theory of residues and the functional equation of $\zeta_\mathbb{K}(s)$ \cite{neukirch}, we see
\begin{align}\label{log der 1}&\frac{1}{2\pi i}\int_{M_T(d)} \frac{\zeta^\prime_\mathbb{K}(s+z)}{\zeta_\mathbb{K}(s+z)}\hat{u}(1+z\log X)\frac{dz}{z}
\\=&\frac{\zeta^\prime_\mathbb{K}(s)}{\zeta_\mathbb{K}(s)}-\sum_{|\gamma|\leq T}\frac{\hat{u}(1-(s-\rho)\log X)}{s-\rho}
-(r_1+r_2)\sum_{m\leq\lfloor d/2\rfloor}\frac{\hat{u}(1-(s+2m)\log X)}{s+2m}\nonumber\\ 
& -r_2\sum_{j\leq\lfloor(d-1)/2\rfloor}\frac{\hat{u}(1-(s+2j+1)\log X)}{s+2j+1}
+\chi_\mathbb{K}\frac{\hat{u}(1-(s-1)\log X)}{s-1}.\nonumber
\end{align}

Since
\[\int_{c-iT}^{c+iT}=\int_{M_T(d)}-\int_{c+iT}^{-d+iT}-\int_{-d+iT}^{-d-iT}-\int_{-d-iT}^{c-iT}\]
it remains to show that these other integrals vanish in the limit of $T$ and $d$. We first consider the integral over the line $(-d\pm iT)$. Now as long as $\sigma$ is negative and bounded away from a negative integer we have 
\begin{equation}\frac{\Gamma^\prime(s)}{\Gamma(s)}\ll \log(|s|+1).\end{equation}
Hence by logarithmic differentiation of the functional equation of $\zeta_\mathbb{K}(s)$ we have 
\begin{equation}\begin{split}\frac{\zeta^\prime_\mathbb{K}(s)}{\zeta_\mathbb{K}(s)}\ll&\frac{\zeta^\prime_\mathbb{K}(1-s)}{\zeta_\mathbb{K}(1-s)}+\frac{\Gamma^\prime(s)}{\Gamma(s)}\\
\ll& 1+\log (|s|+1),
\end{split}\end{equation}
and as such
\begin{equation}\frac{\zeta^\prime_\mathbb{K}(z+s)}{\zeta_\mathbb{K}(z+s)}\ll \log(|z+s|+1).\end{equation}
Hence, if $d$ is a half integer  
\begin{equation}\int_{-d+iT}^{-d-iT}\frac{\zeta^\prime_\mathbb{K}(s+z)}{\zeta_\mathbb{K}(s+z)}\hat{u}(1+z\log X)\frac{dz}{z}\ll T\frac{\log (|d+s|+1)}{|d|}\frac{\max(u(x))}{(d\log X+1)} 
\end{equation}
and this vanishes as $d\to \infty$ through the half integers.

The behaviours of the other two integrals are equivalent so we only consider the case in the upper half-plane.  We split the line $(-d+iT,c+iT)$ at the point $b+iT$ where $b=-1-\Re(s)$. Then similarly to the above we have

\begin{equation}\begin{split}\int_{-d+iT}^{b+iT}\frac{\zeta^\prime_\mathbb{K}(s+z)}{\zeta_\mathbb{K}(s+z)}\hat{u}(1+z\log X)\frac{dz}{z}&\ll\frac{\log |T+s|}{T}\int_{-d}^b \hat{u}(1+y\log X)dy\\&\ll_{X,s}\frac{\log T}{T}.
\end{split}\end{equation}

For the integral over the line $(b+iT,\,c+iT)$ we restrict $T$ in such a way that $|T-\gamma|^{-1}\ll \log T$. Then by combining  Lemmas \ref{N_1 lem} and \ref{log der lem} we have 
  
\begin{equation}\int_{b+iT}^{c+iT}\frac{\zeta^\prime_\mathbb{K}(s+z)}{\zeta_\mathbb{K}(s+z)}\hat{u}(1+z\log X)\frac{dz}{z}\ll_X \frac{\log^2 T}{T}.
\end{equation}

If we vary $T$ by a bounded amount then the sum over zeros in (\ref{log der 1}) incurs $O(\log T)$ extra terms. These terms are all $O_{X,s}(T^{-1})$ so if we want to relax the restriction on $T$ we must take an error of $O(T^{-1}\log T)$. Since this is less than our main error term we can let $T\to \infty$ after $d$.   
\end{proof}

The support condition on $u$ implies $v(e^{\log\mathfrak{N}(\mathfrak{a})/\log X})=0$ when $\mathfrak{N}(\mathfrak{a})>X$. Since there are at most $n$ prime ideals above the rational prime $p$ we see the sum over $\mathfrak{a}\subseteq\mathcal{O}_\mathbb{K}$ is indeed finite. Also, similarly to \cite{hybrid}, we can show the sums over $\rho$, $m$ and $j$ converge absolutely so long as $s\neq\rho$, $s\neq -2m$ or $s\neq -(2j+1)$. We now turn to the proof of Theorem \ref{hybrid prod thm}.

Let $f_\mathbb{K}(n)$ represent the number of ideals of $\mathcal{O}_\mathbb{K}$ with norm $n$. Then
\begin{equation} \zeta_\mathbb{K}(s)=\sum_{n=1}^\infty \frac{f_\mathbb{K}(n)}{n^s}=1+\sum_{n=2}^\infty \frac{f_\mathbb{K}(n)}{n^s}\end{equation}
and so $\zeta_\mathbb{K}(\sigma+it)\to 1$ as $\sigma\to\infty$ uniformly in $t$. Integrating (\ref{log der}) along the horizontal line from $s_0=\sigma_0+it_0$ to $+\infty$, with $\sigma_0\geq 0$ and $|t_0|\geq 2$, we get on the left hand side $-\log \zeta_\mathbb{K}(s_0)$. 
We can now follow the arguments in \cite{hybrid} to find 
\begin{equation}\zeta_\mathbb{K}(s)=\tilde{P}_\mathbb{K}(s,X)Z_\mathbb{K}(s,X)\left(1+O\left(\frac{X^{l+2}}{(|s|\log X)^l}\right)\right)\end{equation}
where 
\begin{equation}\tilde{P}_\mathbb{K}(s,X)=\exp\left(\sum_{\mathfrak{a}\subseteq\mathcal{O}_\mathbb{K}}\frac{\Lambda(\mathfrak{a})}{\mathfrak{N}(\mathfrak{a})^s\log\mathfrak{N}(\mathfrak{a})}v(e^{\log\mathfrak{N}(\mathfrak{a})/\log X})\right).\end{equation}
We note that this is not too different to $P_\mathbb{K}(s,X)$. Indeed, since $v(e^{\log\mathfrak{N}(\mathfrak{a})/\log X})=1$
for $\mathfrak{N}(\mathfrak{a})\leq\ X^{1-1/X}$ we have 
\begin{align*}\tilde{P}_\mathbb{K}(s,X)=&P_\mathbb{K}(s,X)\exp\left(\sum_{\mathfrak{a}\subseteq\mathcal{O}_\mathbb{K}}\frac{\Lambda(\mathfrak{a})}{\mathfrak{N}(\mathfrak{a})^s\log\mathfrak{N}(\mathfrak{a})}(v(e^{\log\mathfrak{N}(\mathfrak{a})/\log X})-1)\right)
\\=&P_\mathbb{K}(s,X)\exp\left(\sum_{X^{1-1/X}\leq\mathfrak{N}(\mathfrak{a})\leq X}\frac{\Lambda(\mathfrak{a})}{\mathfrak{N}(\mathfrak{a})^s\log\mathfrak{N}(\mathfrak{a})}(v(e^{\log\mathfrak{N}(\mathfrak{a})/\log X})-1)\right)
\\=&P_\mathbb{K}(s,X)\exp\left(O\left(\sum_{X^{1-1/X}\leq p\leq X}p^{-\sigma}\right)\right)
\\=&P_\mathbb{K}(s,X)\exp\left(O\left(X^{-\sigma}\log X\right)\right)
\\=&P_\mathbb{K}(s,X)(1+O(X^{-\sigma}\log X),
\end{align*}
where we have again used the fact that at most $n$ prime ideals lie above the rational prime $p$.

To remove the restriction on $s$, we note that we may interpret $\exp(−U(z))$ to be asymptotic
to $Cz$ for some constant $C$ as $z \to 0$, so both sides of (\ref{hybrid prod}) vanish at the zeros.

\section{Moments of the arithmetic factor}\label{p moment sec}

In this section we prove Theorem \ref{arithmetic moments thm}. For a rational prime $p$ we have the decomposition
\begin{equation}p{\mathcal{O}_\mathbb{K}}=\prod_{i=1}^{g}\mathfrak{p}_i^{e_i}
\end{equation}
with 
\begin{equation}\mathfrak{N}(\mathfrak{p}_i)=p^{f_i}
\end{equation}
where $e_i$ and $f_i$ are positive integers. Since $\mathbb{K}$ is Galois, $e_1=e_2=\cdots=e_{g}=e$ and $f_1=f_2=\cdots=f_{g}=f$, say. We then have the identity $efg=n$. Let $g_p$ denote the number of prime ideals lying above $p$. Then
\begin{align}P_\mathbb{K}(s,X)^k=&\exp\Bigg(k\sum_{\mathfrak{N}(\mathfrak{a})\leq X}\frac{\Lambda(\mathfrak{a})}{\mathfrak{N}(\mathfrak{a})^s\log\mathfrak{N}(\mathfrak{a})}\Bigg)\nonumber
=\exp\Bigg(k\sum_m\sum_{\mathfrak{N}(\mathfrak{p})^m\leq X}\frac{1}{m\mathfrak{N}(\mathfrak{p})^{ms}}\Bigg)
\\=&\exp\Bigg(k\sum_m\sum_{g|n}g\sum_{e|\frac{n}{g}}\sum_{\substack{p^{\frac{mn}{eg}}\leq X\\g_p=g}}\frac{1}{mp^{(n/eg)ms}}\Bigg)
\\=&\prod_{g|n}\prod_{e|\frac{n}{g}}\prod_{\substack{p^{\frac{n}{eg}}\leq X\\g_p=g}}\exp\bigg(\log(1-p^{-(n/eg)s})^{-gk}
-\sum_{\substack{m\\p^{\frac{mn}{eg}}>X}}\frac{1}{mp^{(n/eg)ms}}\bigg)\nonumber. 
\end{align}
We now write the innermost product as the Dirichlet series
\begin{equation}\sum_{l\in\mathcal{L}_{e,g}(X)}^\infty \frac{\beta_{gk}(l)}{l^{(n/eg)s}}
\end{equation} 
where $\mathcal{L}_{e,g}(X)=\{l\in\mathrm{Im}(\mathfrak{N}):p|l\implies g_p=g\,\,\text{and}\,\,p^{n/eg}\leq X\}$. We see that 
$\beta_{gk}(l)$ is a multiplicative function of $l$, $0\leq \beta_{gk}(l)\leq d_{gk}(l)$ for all $l$ and $\beta_{gk}(p^m)=d_{gk}(p^m)$ if $p^m\leq X$. 

For an integer $l$, let $l_{e,g}$ denote the greatest factor of $l$ composed of primes $p$ for which $g_p=g$ and whose ramification index is $e$. Now,
\begin{equation}P_\mathbb{K}(s,X)^k=\prod_{g|n}\prod_{e|\frac{n}{g}}\bigg(\sum_{l\in\mathcal{L}_{e,g}(X)}^\infty \frac{\beta_{gk}(l)}{l^{(n/eg)s}}\bigg)=\sum_{l\in\mathcal{W}(X)}^\infty \frac{\gamma_{k}(l)}{l^s}
\end{equation}
where 
\begin{equation}\gamma_k(l)=\prod_{g|n}\prod_{e|\frac{n}{g}}\beta_{gk}(l_{e,g}^{eg/n})
\end{equation}
and $\mathcal{W}(X)=\{l\in\mathrm{Im}(\mathfrak{N}):\mathfrak{N}(\mathfrak{p})|l\implies \mathfrak{N}(\mathfrak{p})\leq X\}$.
The product representation of $\gamma$ is made possible by the fact that for integers $l,m$ belonging to different $\mathcal{L}_{e,g}(X)$, we have $(l,m)=1$. This would not necessarily be the case for non-Galois extensions. For example, in a cubic extension we may have the factorisation $p\mathcal{O}_\mathbb{K}=\mathfrak{p}_1\mathfrak{p}_2$ and hence one of these ideals has norm $p$, whilst the other has norm $p^2$. We could then follow the previous reasoning whilst redefining the sets $\mathcal{L}$ with a consideration of this difference. However, we would then lose the coprimality condition.  

Since we want to apply the mean value theorem for Dirichlet series we split the sum at $T^\theta$ where $\theta$ is to be chosen later and obtain
\begin{equation}\label{P trunc}P_\mathbb{K}(s,X)^k=\sum_{\substack{l\in\mathcal{W}(X)\\l\leq T^\theta}}\frac{\gamma_k(l)}{l^{s}}+O\bigg(\sum_{\substack{l\in\mathcal{W}(X)\\l> T^\theta}}\frac{\gamma_k(l)}{l^{s}}\bigg).\end{equation}
Now for $\epsilon>0$ and $\sigma\geq c$ the error term is 
\begin{align}\ll&T^{-\epsilon\theta}\sum_{l\in\mathcal{W}(X)}\frac{\prod_{g|n}\prod_{e|\frac{n}{g}}d_{gk}(l_{e,g}^{eg/n})}{n^{c-\epsilon}}
=T^{-\epsilon\theta}\prod_{\mathfrak{N}(\mathfrak{p})\leq X}(1-\mathfrak{N}(\mathfrak{p})^{\epsilon-c})^{-k}\nonumber
\\=&T^{-\epsilon\theta}\exp\bigg(O\bigg(k\sum_{\mathfrak{N}(\mathfrak{p})\leq X}\mathfrak{N}(\mathfrak{p})^{\epsilon-c}\bigg)\bigg)
=T^{-\epsilon\theta}\exp\left(O\left(\frac{kX^{1-c+\epsilon}}{(1-c+\epsilon)\log X}\right)\right)\nonumber
\end{align} 
where in the last line we have used the prime ideal theorem. If we let $ X\asymp (\log T)^{1/(1-c+\epsilon)}$ then this is 
\begin{equation}\ll T^{-\epsilon\theta}\exp\left(O\left(k\frac{\log T}{\log\log T}\right)\right)\ll_k T^{-\epsilon\theta/2}\end{equation}
and hence
\begin{equation}\label{P trunc 1}P_\mathbb{K}(s,X)^k
=\sum_{\substack{l\in\mathcal{W}(X)\\l\leq T^\theta}}\frac{\gamma_k(l)}{l^s}+O_k(T^{-\epsilon\theta/2}).
\end{equation}
We now let $\theta=1/2$ and apply the Montgomery-Vaughan mean value theorem \cite{mv 0} to give
\begin{align}\frac{1}{T}\int_T^{2T}\bigg|\sum_{\substack{l\in\mathcal{W}(X)\\l\leq T^{1/2}}}\frac{\gamma_k(l)}{l^{\sigma+it}}\bigg|^2dt=&(1+O(T^{-1/2}))\sum_{\substack{l\in\mathcal{W}(X)\\l\leq T^{1/2}}}\frac{\gamma_k(l)^2}{l^{2\sigma}}\nonumber
\\=&(1+O(T^{-1/2}))\bigg(\sum_{l\in\mathcal{W}(X)}\frac{\gamma_k(l)^2}{l^{2\sigma}}+O(T^{-\epsilon/4})\bigg)
\\=&(1+O(T^{-\epsilon/4}))\sum_{l\in\mathcal{W}(X)}\frac{\gamma_k(l)^2}{l^{2\sigma}}.\nonumber
\end{align}
Therefore by (\ref{P trunc 1}) and the Cauchy-Schwarz inequality we have 
\begin{equation}\frac{1}{T}\int_T^{2T}\left|P_\mathbb{K}\left(\sigma+it,X\right)\right|^{2k}=(1+O(T^{-\epsilon/4}))\sum_{l\in\mathcal{W}(X)}\frac{\gamma_k(l)^2}{l^{2\sigma}}.\end{equation}
We can now re-factorise the above Dirichlet series to give
\begin{equation}\label{gamma prod rep}\sum_{l\in\mathcal{W}(X)}\frac{\gamma_k(l)^2}{l^{2\sigma}}=\prod_{g|n}\prod_{e|\frac{n}{g}}\bigg(\sum_{l\in\mathcal{L}_{e,g}(X)}^\infty \frac{\beta_{gk}(l)^2}{l^{2(n/eg)\sigma}}\bigg).
\end{equation}
For an individual series in the above product we can follow the arguments in \cite{hybrid} to find 
\begin{equation}\sum_{l\in\mathcal{L}_{e,g}(X)}^\infty \frac{\beta_{gk}(l)^2}{l^{2(n/eg)\sigma}}=(1+O(X^{-1/2+\epsilon}))\prod_{\substack{p^{\frac{n}{eg}}\leq X\\g_p=g}}\sum_{m\geq 0} \frac{d_{gk}(p^m)^2}{p^{2m(n/eg)\sigma}}.
\end{equation}
Now, the above product may be divergent as $X\to\infty$. In order to keep the arithmetic information, we factor out the divergent part and write it as 
\begin{equation}\label{div conv prod}\prod_{\substack{p^{\frac{n}{eg}}\leq X\\g_p=g}}\left(\left( 1-p^{-2(n/eg)\sigma}\right)^{ngk^2}\sum_{m\geq 0}\frac{d_{gk}(p^m)^2}{p^{2m(n/eg)\sigma}}\right)\prod_{\substack{p^{\frac{n}{eg}}\leq X\\g_p=g}}\left( 1-p^{-2(n/eg)\sigma}\right)^{-ngk^2}.\end{equation}
In terms of divergence, the worst case scenario is when $n/eg=1$. If in this case $g_p=g<n$, then $p$ is ramified and hence the product is finite. Therefore, we only need consider the case $g=n$, for which the above equals
\begin{align}\prod_{\substack{p>X\\g_p=n}}\left(\left(1-p^{-2\sigma}\right)^{n^2k^2}\sum_{m\geq 0}\frac{d_{nk}(p^m)^2}{p^{2m\sigma}}\right)
=&\prod_{\substack{p>X\\g_p=n}}\left(1-n^2k^2p^{-2\sigma}+n^2k^2p^{-2\sigma}+O_k(p^{-4\sigma})\right)\nonumber
\\=&\prod_{\substack{p>X\\g_p=n}}\left(1+O_k(p^{-4\sigma})\right)\nonumber
\\=&1+O_k(1/(X\log X)).
\end{align}
It follows that we can extend the first product in (\ref{div conv prod}) over all primes. Specialising to $\sigma=1/2$ and using the  product representation in (\ref{gamma prod rep}) we see 
\begin{equation}\sum_{l\in\mathcal{W}(X)}\frac{\gamma_k(l)^2}{l}=a(k)\prod_{\mathfrak{N}(\mathfrak{p})\leq X}(1-\mathfrak{N}(\mathfrak{p})^{-1})^{-nk^2}(1+O_k(X^{-1/2+\epsilon})).\end{equation}
By a generalisation of Mertens theorem \cite{mertens}, we have
\begin{equation}\prod_{\mathfrak{N}(\mathfrak{p})\leq X}(1-\mathfrak{N}(\mathfrak{p})^{-1})^{-nk^2}=\chi_\mathbb{K}^{nk^2}(e^{\gamma}\log X)^{nk^2}(1+O(1/\log^2 X))\end{equation}
and the result follows.

\section{Support for conjecture \ref{z moment conj}} 
Let $\mathbb{K}$ be a Galois extension of degree $n$ with Galois group $G$. Then it is well known (see for example \cite{neukirch}, chap.\ 7) that 
\begin{equation}\label{gal ded}\zeta_\mathbb{K}(s)=\prod_\chi L(s,\chi,\mathbb{K}/\mathbb{Q})^{\chi(1)}
\end{equation} 
where the product is over the non-equivalent irreducible characters of $G$ and $L(s,\chi,\mathbb{K}/\mathbb{Q})$ is the Artin $L$-function attached to $\chi$. For each character $\chi$, the associated $L$-function  satisfies the functional equation
\begin{equation}\Lambda(s,\chi):=q(\chi)^{s/2}\gamma(s,\chi)L(s,\chi)=W(\chi)\Lambda(1-s,\overline{\chi})
\end{equation}
where $W(\chi)$ is some complex number of modulus one  and $q(\chi)$ is the conductor, for which we do not require an explicit expression. The gamma factor is given by 
\begin{equation}\gamma(s,\chi)=\pi^{-sd_\chi/2}\prod_{j=1}^{d_\chi}\Gamma\left(\frac{s+\mu_j}{2}\right)
\end{equation}
with $\mu_j$ equal to 0 or 1. If we assume the Artin conjecture then $L(s,\chi)$ is an entire function for all non-trivial $\chi$. If $\chi$ is the trivial character then $L(s,\chi)$ equals the Dedekind zeta function of the base field, which in our case is $\zeta(s)$. Under this assumption, these $L$-functions exhibit reasonable behaviour and the usual arguments  (e.g. Theorem 5.8 of \cite{iwaniec kowalski}) give the mean density of zeros of $L(\beta+it,\chi)$, $0\leq\beta\leq 1$, as 
\begin{equation}\frac{1}{\pi}\log\bigg(q(\chi)\left(\frac{t}{2\pi}\right)^{d_\chi}\bigg)=\frac{1}{\pi}\mathcal{L}_\chi(t),
\end{equation}
say. For each $L(s,\chi)$ in the product of equation (\ref{gal ded}), we associate to its zeros $\gamma_n(\chi)$ at height $T$, a unitary matrix $U(N(\chi))$ of size $N(\chi)=\lfloor\mathcal{L}_\chi(T)\rfloor$ chosen with respect to Haar measure, which we denote $d\mu(\chi)$. 
After rescaling, the zeros $\gamma_n(\chi)$ are conjectured \cite{rud sarn} to share the same distribution as the eigenangles $\theta_n(\chi)$ of $U(N(\chi))$ when chosen with $d\mu(\chi)$. 

In addition to the previous assumptions, we now also assume the extended Riemann hypothesis. Let $Z_\mathbb{K}(s,X)$ be given by (\ref{Z prod}). Since $\Re E_1(ix)=-\mathrm{Ci}(|x|)$ for $x\in\mathbb{R}$, where
\begin{equation}\mathrm{Ci}(z)=-\int_z^\infty \frac{\cos w}{w}dw,
\end{equation}
we see that
\begin{equation}\begin{split}\frac{1}{T}\int_T^{2T}&\left|Z_\mathbb{K}\left(\frac{1}{2}+it,X\right)\right|^{2k}dt
\\=&\frac{1}{T}\int_T^{2T}\prod_{\gamma_n}\exp\bigg(2k\int_1^e u(y)\mathrm{Ci}(|t-\gamma_n|\log y\log X)\bigg)dydt
\\=&\frac{1}{T}\int_T^{2T}\prod_\chi\prod_{\gamma_n(\chi)}\exp\bigg(2k\chi(1)\int_1^e u(y)\mathrm{Ci}(|t-\gamma_n(\chi)|\log y\log X)\bigg)dydt
\end{split}
\end{equation}
where $u(y)$ is a smooth, non-negative function supported on $[ e^{1-1/X},e]$ and of total mass one. We now replace the zeros with the eigenangles and argue that the above should be modeled by 
\begin{equation}\mathbb{E}\bigg[\prod_\chi\prod_{n=1}^{N(\chi)}\phi({k\chi(1)},\theta_n(\chi))\bigg]
\end{equation}
where
\begin{equation}\phi(m,\theta)=\exp\bigg(2m\int_1^e u(y)\mathrm{Ci}(|\theta|\log y\log X)\bigg)
\end{equation}
and the expectation is taken with respect to the product measure $\prod_\chi d\mu(\chi)$. We now assume that the matrices $U(N(\chi))$ can be chosen independently for any two distinct $\chi$. This corresponds to a `superposition' of ensembles; the behaviour of which is also shared by the distribution of zeros of a product of distinct $L$-functions \cite{liu ye}. With this assumption, the expectation factorises as 
\begin{equation}\prod_\chi \mathbb{E}\bigg[\prod_{n=1}^{N(\chi)}\phi(k\chi(1),\theta_n(\chi))\bigg].
\end{equation}    
In \cite{hybrid} it is shown (Theorem 4) that for $k>-1/2$ and $X\geq 2$,
\begin{equation}\mathbb{E}\bigg[\prod_{j=1}^{M}\phi(m,\theta_j)\bigg]\sim\frac{G(m+1)^2}{G(2m+1)}\bigg(\frac{M}{e^\gamma\log X}\bigg)^{m^2}\left(1+O_m\left(\frac{1}{\log X}\right)\right).
\end{equation}
Therefore, by forming the product over $\chi$ and using $\sum_{\chi}\chi(1)^2=|\mathrm{Gal}(\mathbb{K}/\mathbb{Q})|$$=n$ we are led to conjecture \ref{z moment conj}.

\section{The second moment of $Z_\mathbb{K}$ for quadratic extensions}\label{z moment sec}

In this section we prove Theorem \ref{z moment thm}. For the most part, the remainder of this paper is concerned with quadratic extensions so we first state some useful facts whilst establishing our notation. 

As mentioned in the introduction, $\zeta_\mathbb{K}=\zeta(s)L(s,\chi)$ where $\chi$ is the Kronecker character. We shall have occasion to work with more general (complex) characters $\chi$ mod $q>1$ when the arguments in question work in such generalities, however, at some points we may specialise to the Kronecker character without mention. 
We also note in quadratic extensions the splitting of primes admits the following simple description:
\begin{eqnarray*}p\,\,\mathrm{is \,\,split}&:&\,\,\,(p)=\mathfrak{p}_1\mathfrak{p}_2\,\,\,\implies\mathfrak{N}(\mathfrak{p}_1)=\mathfrak{N}(\mathfrak{p}_2)=p
\\p\,\,\mathrm{is \,\,inert}&:&\,\,\,(p)=\mathfrak{p}_1\,\,\,\,\,\,\,\,\implies\mathfrak{N}(\mathfrak{p}_1)=p^2
\\p\,\,\mathrm{is \,\,ramified}&:&\,\,\,(p)=\mathfrak{p}_1^2\,\,\,\,\,\,\,\,\implies\mathfrak{N}(\mathfrak{p}_1)=p.
\end{eqnarray*}
At some points we shall use the notation $p_\mathrm{s}, p_\mathrm{i}, p_\mathrm{r}$ to denote split, inert and ramified primes respectively.

\subsection{The setup}
Our aim is to show 
\begin{equation}\frac{1}{T}\int_T^{2T}\bigg|Z_\mathbb{K}\left(\frac{1}{2}+it,X\right)\bigg|^2dt\sim\frac{\log T\cdot\log qT}{\left(e^\gamma \log X\right)^2}
\end{equation}
for $X,T\to\infty$ with $X\ll (\log T)^{2-\epsilon}$ and $\mathbb{K}$ quadratic. Since $\zeta_\mathbb{K}(1/2+it)P_\mathbb{K}(1/2+it,X)=Z_\mathbb{K}(1/2+it,X)(1+o(1))$ for $t\in[T,2T]$, it is enough to show that 
\begin{equation}\frac{1}{T}\int_T^{2T}\bigg|\zeta_\mathbb{K}\left(\frac{1}{2}+it\right)P_\mathbb{K}\left(\frac{1}{2}+it,X\right)^{-1}\bigg|^2dt\sim\frac{\log T\cdot\log qT}{\left(e^\gamma \log X\right)^2}.
\end{equation}
To evaluate the left hand side we first express $P_\mathbb{K}(1/2+it)^{-1}$ as a Dirichlet polynomial and then apply a formula given given by the author in \cite{me}. The means to do this are given by the following sequence of Lemmas.

\begin{lem}\label{Q lem}Let 
\begin{equation}\label{Q split prod}Q_\mathrm{s}(s,X)=\prod_{\substack{p\leq \sqrt{X}\\p\,\,\mathrm{ split}}}\left(1-p^{-s}\right)\prod_{\substack{\sqrt{X}<p\leq X\\p\,\,\mathrm{split}}}\left(1-p^{-s}+\frac{1}{2}p^{-2s}\right)
\end{equation}
and define $Q_{\mathrm{i}}(s,X)$ and $Q_{\mathrm{r}}(s,X)$ as the same products except over the inert and ramified primes respectively. Then for $X$ sufficiently large, we have
\begin{equation}\label{P as Q prod}P_\mathbb{K}(s,X)^{-1}=Q_\mathrm{s}(s,X)^2Q_\mathrm{i}(2s,\sqrt{X})Q_\mathrm{r}(s,X)\left(1+O\left(\frac{1}{\log X}\right)\right)
\end{equation}
and this holds uniformly for $\sigma\geq 1/2$.
\end{lem}
\begin{proof}First, note
\begin{align}
P_\mathbb{K}(s,X)^{-1}=&\exp\left(-\sum_{\substack{\mathfrak{N}(\mathfrak{p})^m\leq X}}\frac{1}{m\mathfrak{N}(\mathfrak{p})^{ms}}\right)\nonumber
\\=&\prod_{\substack{p\leq X\\p\,\,\mathrm{split}}}\exp\Bigg(-2\sum_{1\leq m\leq \left\lfloor\frac{\log X}{\log p}\right\rfloor}\frac{1}{mp^{ms}}\Bigg)
\prod_{\substack{p^2\leq X\\p\,\,\mathrm{inert}}}\exp\Bigg(-\sum_{1\leq m\leq \left\lfloor\frac{\log X}{2\log p}\right\rfloor}\frac{1}{mp^{2ms}}\Bigg)
\\&\,\,\,\,\,\,\,\,\,\,\,\,\,\,\,\,\,\,\,\,\,\,\,\,\times \prod_{\substack{p\leq X\\p\,\,\mathrm{ramified}}}\exp\Bigg(-\sum_{1\leq m\leq \left\lfloor\frac{\log X}{\log p}\right\rfloor}\frac{1}{mp^{ms}}\Bigg)\nonumber
\end{align}
and so it suffices to consider just one of these products. Let $A$ be a subset of the primes and let $N_p=\lfloor \log X/\log p\rfloor$. Since $N_p=1$ if $\sqrt{X}<p\leq X$ we have 
\begin{align*}\prod_{\substack{p\leq X\\p\in A}}\exp\Bigg(-\sum_{1\leq m\leq N_p}\frac{1}{mp^{ms}}\Bigg)=&\prod_{\substack{p\leq \sqrt{X}\\p\in A}}\exp\Bigg(\log(1-p^{-s})+\sum_{m> N_p}\frac{1}{mp^{ms}}\Bigg)
\\&\times\prod_{\substack{\sqrt{X}<p\leq X\\p\in A}}\exp(-p^{-s}).
\end{align*}
Now, on noting that $N_p+1>\log X/\log p$ we have for $\sigma\geq 1/2$; 
\begin{align}\exp\Bigg(\sum_{p\leq \sqrt{X}}\sum_{m>N_p}\frac{1}{mp^{ms}}\Bigg)\ll\exp\Bigg(\sum_{p\leq \sqrt{X}}\frac{1}{p^{\sigma(N_p+1)}}\Bigg)\ll\exp\Bigg(X^{-1/2}\sum_{p\leq \sqrt{X}}1\Bigg)
\end{align}
and this is $\ll 1+O(1/\log X)$ by the prime number theorem. Also, 
\begin{align}&\prod_{\substack{\sqrt{X}<p\leq X}}\left(1-p^{-s}+\frac{1}{2!}p^{-2s}-\frac{1}{3!}p^{-3s}+O(p^{-4\sigma})\right)\nonumber
\\=&\prod_{\substack{\sqrt{X}<p\leq X}}\left(1-p^{-s}+\frac{1}{2}p^{-2s}\right)\left(1+O(p^{-3\sigma})\right)
\\=&\prod_{\substack{\sqrt{X}<p\leq X}}\left(1-p^{-s}+\frac{1}{2}p^{-2s}\right)\left(1+O\left(\frac{1}{\log X}\right)\right)\nonumber
\end{align}
and so we're done. 
\end{proof}

\begin{lem}\label{alpha lem}We have 
\begin{equation}P_\mathbb{K}\left(\frac{1}{2}+it,X\right)^{-1}=\left(1+O\left(\frac{1}{\log X}\right)\right)\sum_{n\in\mathcal{W}(X)}\frac{\alpha(n)}{n^{1/2+it}}
\end{equation}
where $\mathcal{W}(X)=\{n\in\mathrm{Im}(\mathfrak{N}):\mathfrak{N}(\mathfrak{p})|n\implies \mathfrak{N}(\mathfrak{p})\leq X\}$ and the behaviour of $\alpha$ at primes is determined by
\begin{equation}\alpha(p_\mathrm{s}^j)=
\begin{cases}-2&\text{if $j=1$,\,\,$p_\mathrm{s}\leq X$},
\\1&\text{if $j=2$,\,\,$p_\mathrm{s}\leq \sqrt{X}$,}
\\2&\text{if $j=2$,\,\,$\sqrt{X}<p_\mathrm{s}\leq X$,}
\\0&\text{if $j\geq 3$,}
\end{cases}
\,\,\,\alpha(p_\mathrm{i}^{2j})=
\begin{cases}-1&\text{if $j=1$,\,\,$p_\mathrm{i}^2\leq X$},
\\0&\text{if $j=2$,\,\,$p^2_\mathrm{i}\leq \sqrt{X}$,}
\\\frac{1}{2}&\text{if $j=2$,\,\,$\sqrt{X}<p^2_\mathrm{i}\leq X$,}
\\0&\text{if $j\geq 3$}
\end{cases}
\end{equation}
and
\begin{equation}\alpha(p_\mathrm{r}^{j})=
\begin{cases}-1&\text{if $j=1$,\,\,$p_\mathrm{r}\leq X$},
\\0&\text{if $j=2$,\,\,$p_\mathrm{r}\leq \sqrt{X}$,}
\\\frac{1}{2}&\text{if $j=2$,\,\,$\sqrt{X}<p_\mathrm{r}\leq X$,}
\\0&\text{if $j\geq 3$.}
\end{cases}
\end{equation}
We also have the bound $\alpha(n)\ll d(n)$ for all $n\in\mathcal{W}(X)$.
\end{lem}
\begin{proof}
We first note that the square of the product over split primes in (\ref{P as Q prod}) is given by 
\begin{align}\label{Q^2}Q_\mathrm{s}(s,X)^2=&\prod_{\substack{p\leq \sqrt{X}\\p\,\,\mathrm{ split}}}\left(1-2p^{-s}+p^{-2s}\right)\prod_{\substack{\sqrt{X}<p\leq X\\p\,\,\mathrm{split}}}\left(1-2p^{-s}+2p^{-2s}+O(p^{-3\sigma})\right)\nonumber
\\=&\prod_{\substack{p\leq \sqrt{X}\\p\,\,\mathrm{ split}}}\left(1-2p^{-s}+p^{-2s}\right)\prod_{\substack{\sqrt{X}<p\leq X\\p\,\,\mathrm{split}}}\left(1-2p^{-s}+2p^{-2s}\right)
\\&\times\left(1+O\left(\frac{1}{\log X}\right)\right).\nonumber
\\=&R_\mathrm{s}(s,X)\left(1+O\left(\frac{1}{\log X}\right)\right)\nonumber,
\end{align}
say. On writing
\begin{equation}R_\mathrm{s}(s,X)Q_\mathrm{i}(2s,\sqrt{X})Q_\mathrm{r}(s,X)=\sum_{n\in\mathcal{W}(X)}\frac{\alpha(n)}{n^{1/2+it}}
\end{equation}
we can read off the behaviour of $\alpha$ at the primes from the Euler products.
\end{proof}

\begin{lem}Let $\theta>0$. Then 
\begin{equation}\begin{split}
&\frac{1}{T}\int_T^{2T}\left|\zeta_\mathbb{K}\left(\frac{1}{2}+it\right)P_\mathbb{K}\left(\frac{1}{2}+it,X\right)^{-1}\right|^2dt
\\&=\left(1+O\left(\frac{1}{\log X}\right)\right)\frac{1}{T}\int_T^{2T}\bigg|\zeta_\mathbb{K}\left(\frac{1}{2}+it\right)\sum_{\substack{n\in\mathcal{W}(X)\\n\leq T^{\theta}}}\frac{\alpha(n)}{n^{1/2+it}}\bigg|^2dt.
\end{split}\end{equation}
\end{lem}
\begin{proof}First, we write
\begin{equation}\label{Q trunc}Q_\mathbb{K}\left(\frac{1}{2}+it\right)=\sum_{\substack{n\in\mathcal{W}(X)\\n\leq T^\theta}}\frac{\alpha(n)}{n^{1/2+it}}+O\Bigg(\sum_{\substack{n\in\mathcal{W}(X)\\n> T^\theta}}\frac{\alpha(n)}{n^{1/2+it}}\Bigg).
\end{equation}
We can show, by using the bound $\alpha(n)\ll d(n)$ and a similar reasoning to that used between (\ref{P trunc}) and (\ref{P trunc 1}), that if $X\ll \log^{2-\epsilon}T$ then the error term is $\ll T^{-\epsilon\theta/10}$. Rewriting (\ref{Q trunc}) as $Q_\mathbb{K}(1/2+it)=\sum+O(T^{-\epsilon\theta/10})$  we see
\begin{multline}\frac{1}{T}\int_T^{2T}\left|\zeta_\mathbb{K}\left(\frac{1}{2}+it\right)Q_\mathbb{K}\left(\frac{1}{2}+it\right)\right|^2dt
\\=\frac{1}{T}\int_T^{2T}\left|\zeta_\mathbb{K}\left(\frac{1}{2}+it\right)\sum\right|^2dt+O\Bigg(\frac{1}{T^{1+\epsilon\theta/10}}\int_T^{2T}\left|\zeta_\mathbb{K}\left(\frac{1}{2}+it\right)\right|^2\left|\sum\right|dt\Bigg)
\\+O\Bigg(\frac{1}{T^{1+\epsilon\theta/5}}\int_T^{2T}\left|\zeta_\mathbb{K}\left(\frac{1}{2}+it\right)\right|^2dt\Bigg).
\end{multline}
The final term is $\ll T^{-\epsilon\theta/10}$ by Motohashi's result (\ref{second moment}). Using the Cauchy-Schwarz inequality we can show that the second term is 
\begin{equation}\begin{split}\ll &\frac{1}{T^{1+\epsilon\theta/10}}\Bigg(\int_T^{2T}\left|\zeta_\mathbb{K}\left(\frac{1}{2}+it\right)\sum\right|^2dt
\int_T^{2T}\left|\zeta_\mathbb{K}\left(\frac{1}{2}+it\right)\right|^2dt\Bigg)^{1/2}
\\\ll&\frac{1}{T^{1/2+\epsilon\theta/20}}\Bigg(\int_T^{2T}\left|\zeta_\mathbb{K}\left(\frac{1}{2}+it\right)\sum\right|^2dt
\Bigg)^{1/2}
\end{split}\end{equation}
and the result follows.
\end{proof}

We are now required to show that for $X,T\to\infty$ with $X\ll (\log T)^{2-\epsilon}$,
\begin{equation}\label{Q twist}\frac{1}{T}\int_T^{2T}\bigg|\zeta_\mathbb{K}\left(\frac{1}{2}+it\right)\sum_{\substack{n\in\mathcal{W}(X)\\n\leq T^\theta}}\frac{\alpha(n)}{n^{1/2+it}}\bigg|^2=\frac{\log T\cdot\log qT}{\left(e^\gamma \log X\right)^2}\left(1+O\left(\frac{1}{\log X}\right)\right).
\end{equation}
In order to state the formula given in \cite{me} we must first establish some notation. So, let 
$\alpha,\beta,\gamma,\delta$ be complex numbers $\ll1/\log T$ and let 
\begin{equation}\label{A}\begin{split}A_{\alpha,\beta,\gamma,\delta}(s)=&\zeta(1+\alpha+\gamma+s)\zeta(1+\beta+\delta+s)L(1+\beta+\gamma+s,\chi)\\&\times \frac{L(1+\alpha+\delta+s,\overline{\chi})}{\zeta(2+\alpha+\beta+\gamma+\delta+2s)}
\prod_{p|q}\left(\frac{1-p^{-1-s-\beta-\delta}}{1-p^{-2-2s-\alpha-\beta-\gamma-\delta}}\right).
\end{split}\end{equation}
For integers $h$ and $k$ let
\begin{equation}B_{\alpha,\beta,\gamma,\delta,h,k}(s,{\chi})=\prod_{p|hk}\frac{\sum_{j\geq 0}f_{\alpha,\beta}(p^{k_p+j},\chi)f_{\gamma,\delta}(p^{h_p+j},\overline{\chi})p^{-j(1+s)}}{\sum_{j\geq 0}f_{\alpha,\beta}(p^{j},\chi)f_{\gamma,\delta}(p^{j},\overline{\chi})p^{-j(1+s)}}
\end{equation}
where
\begin{equation}f_{\alpha,\beta}(n,\chi)=\sum_{n_1n_2=n}n_1^{-\alpha}n_2^{-\beta}\chi(n_2)
\end{equation} 
and where $h_p$ and $k_p$ are the highest powers of $p$ dividing $h$ and $k$ respectively.  
Now let 
\begin{equation}Z_{\alpha,\beta,\gamma,\delta,h,k}(s)=A_{\alpha,\beta,\gamma,\delta}(s)B_{\alpha,\beta,\gamma,\delta,h,k}(s,\chi).
\end{equation}
We must also define a slight variant of the above. For this we let
\begin{multline}\label{A prime}A_{\alpha,\beta,\gamma,\delta}^\prime (s,\chi)\\=\frac{L(1+\alpha+\gamma+s,\chi)L(1+\beta+\delta+s,\chi)L(1+\alpha+\delta+s,\chi)L(1+\beta+\gamma+s,\chi)}{L(2+\alpha+\beta+\gamma+\delta+2s,\chi^2)}
\end{multline}
and
\begin{equation}B^\prime_{\alpha,\beta,\gamma,\delta,h,k}(s,\chi)=\prod_{p|hk}\frac{\sum_{j\geq 0}\chi(p^j)\sigma_{\alpha,\beta}(p^{k_p+j})\sigma_{\gamma,\delta}(p^{h_p+j})p^{-j(1+s)}}{\sum_{j\geq 0}\chi(p^j)\sigma_{\alpha,\beta}(p^{j})\sigma_{\gamma,\delta}(p^{j})p^{-j(1+s)}}.
\end{equation}
where
\begin{equation}\sigma_{\alpha,\beta}(n)=\sum_{n_1n_2=n}n_1^{-\alpha}n_2^{-\beta}.
\end{equation}
Now let
\begin{equation}Z^\prime_{\alpha,\beta,\gamma,\delta,h,k}(s,\chi)=\overline{G(\chi)}A^\prime_{\alpha,\beta,\gamma,\delta}(s,\chi)B^\prime_{\alpha,\beta,\gamma,\delta,h,k}(s,\chi)
\end{equation}
where $G(\chi)$ is the Gauss sum associated to $\chi$.

\begin{thm}[\cite{me}]\label{twisted thm}
Let
\begin{equation}\begin{split}I(h,k)=&\int_{-\infty}^\infty\left(\frac{h}k\right)^{-it}\zeta\left(\frac{1}2+\alpha+it\right)L\left(\frac{1}2+\beta+it,\chi\right)\\&\times\zeta\left(\frac{1}2+\gamma-it\right)L\left(\frac{1}2+\delta-it,\overline{\chi}\right)w(t)dt
\end{split}\end{equation}
where $w(t)$ is a smooth, nonnegative function with support contained in $[T/2,4T]$, satisfying $w^{(j)}(t)\ll_j T_0^{-j}$ for all $j=0,1,2,\ldots,$ where $T^{1/2+\epsilon}\ll T_0 \ll T$. Suppose $(h,k)=1$ and that $hk\leq T^{\frac{2}{11}-\epsilon}$. Then
\begin{equation}\label{main thm form}\begin{split}
I(h,k)=&\frac{1}{\sqrt{hk}}\int_{-\infty}^\infty w(t)\bigg(Z_{\alpha,\beta,\gamma,\delta,h,k}(0)+\frac{1}{q^{\beta+\delta}} Z_{-\gamma,-\delta,-\alpha,-\beta,h,k}(0)\left(\frac{t}{2\pi}\right)^{-\alpha-\beta-\gamma-\delta}
\\&+Z_{-\gamma,\beta,-\alpha,\delta,h,k}(0)\left(\frac{t}{2\pi}\right)^{-\alpha-\gamma}+\frac{1}{q^{\beta+\delta}}Z_{\alpha,-\delta,\gamma,-\beta,h,k}(0)\left(\frac{t}{2\pi}\right)^{-\beta-\delta}
\\&+{\bf 1}_{q|h}\frac{\chi(k)}{q^\delta}Z^\prime_{-\delta,\beta,\gamma,-\alpha,\frac{h}q,k}(0,\chi)\left(\frac{t}{2\pi}\right)^{-\alpha-\delta} 
\\&+{\bf 1}_{q|k}\frac{\overline{\chi}(h)}{q^\beta} Z^\prime_{\alpha,-\gamma,-\beta,\delta,h,\frac{k}q}(0,\overline{\chi})\left(\frac{t}{2\pi}\right)^{-\beta-\gamma}\bigg)dt+E(T)
\end{split}\end{equation}
where
\begin{equation}E(T)\ll T^{3/4+\epsilon}(hk)^{7/8+\epsilon}q^{1+\epsilon}(T/T_0)^{9/4}.
\end{equation}
\end{thm}
Now take a Dirichlet polynomial $M(s)=\sum_{n\leq T^\theta}a(n)n^{-s}$ with $\theta\leq 1/11-\epsilon$ and let $w(t)$ satisfy the conditions of Theorem \ref{twisted thm}. Then, upon expanding, we have 
\begin{multline}\int_{-\infty}^{\infty}\left|\zeta_{\mathbb{K}}\left(\frac{1}2+it\right)\right|^2\left|M\left(\frac{1}2+it\right)\right|^2w(t)dt\\=\sum_{h,k\leq T^\theta}\frac{a(h)\overline{a(k)}}{\sqrt{hk}}(h,k)\lim_{\alpha,\beta,\gamma,\delta\to 0}I(h_k,k_h)
\end{multline}
where $h_k=h/(h,k)$. In order to evaluate this inner limit we express $Z_{\alpha,\beta,\gamma,\delta,h,k}(0)$ as a Laurent series and express the other terms as Taylor series. In doing this, the only real difficulty lies in calculating the derivatives of $B_ {\alpha,\beta,\gamma,\delta,h,k}(0)$. For our purposes, which is to work over $X$-smooth numbers, we only need upper bounds however. The first order derivatives of $B_{\alpha,\beta,\gamma,\delta,h,k}(0)$ are
\begin{equation}\ll B_{0,0,0,0,h,k}(0)\bigg(\log hk+\sum_{p|hk}\frac{\log p}{p}\bigg)\ll B_{0,0,0,0,h,k}(0)\big(\log hk+\log\log hk\big).
\end{equation} 
Similarly, one finds that the second order derivatives are 
\begin{equation}\ll B_{0,0,0,0,h,k}(0)\big(\log^2 hk+\log hk\log\log hk+\log^2\log hk\big).
\end{equation} 
A short calculation gives
\begin{equation}
B_{0,0,0,0,h,k}(0)=\delta(h)\delta(k)
\end{equation}
where 
\begin{equation}\delta(h)=\begin{cases}\prod_{\substack{p|h\\p \,\,\,\mathrm{split}}}\left(1+h_p\frac{1-p^{-1}}{1+p^{-1}}\right)& \text{{if $h_{\mathrm{i}}$ is square}}\\ 0&\text{otherwise}\end{cases}
\end{equation}
and $h_{\mathrm{i}}$ is the greatest factor of $h$ composed solely of inert primes. 

Upon taking the limit as $\alpha,\beta,\gamma,\delta\to 0$ and taking smooth approximations to the characteristic function of the interval $[T,2T]$ with $T_0=T^{1-\epsilon}$ we get the following

\begin{prop}\label{my thm}
Let $M(s)=\sum_{n\leq T^\theta}a(n)n^{-s}$ with $\theta\leq 1/11-\epsilon$. Then,
\begin{multline}\label{ded twist}\frac{1}{T}\int_T^{2T}\left|\zeta_{\mathbb{K}}\left(\frac{1}2+it\right)\right|^2\left|M\left(\frac{1}2+it\right)\right|^2dt\\=\sum_{h,k\leq T^\theta}\frac{a(h)\overline{a(k)}}{hk}(h,k)\Bigg[ \sum_{n=0}^2c_n(h,k,T)
+O\left( T^{-\frac{1}{4}+\epsilon}\left(h_kk_h\right)^{7/8+\epsilon}\right)\Bigg].
\end{multline}
The leading order term is given by
\begin{multline}\label{c_2}c_2(h,k,T)= \frac{6}{\pi^2}L(1,\chi)^2\prod_{p|d_\mathbb{K}}\left(1+\frac{1}{p}\right)^{-1}
\\\times\delta(h_{k})\delta(k_{h})\bigg[\log T\cdot\log qT+O(\log T\log h_kk_h)\bigg].
\end{multline}
 For the lower order terms we have \begin{equation}\label{c_1}c_1(h,k,T)\ll \delta(h_{k})\delta(k_{h})\log T\log\log h_kk_h\end{equation}
and 
\begin{equation}\begin{split}\label{c_0}c_0(h,k,T)=&c^\prime_0(h,k,T)+{\bf 1}_{q|h_k}{\chi(k_h)}Z^\prime_{0,0,0,0,\frac{h_k}q,k_h}(0,\chi)
\\&+{\bf 1}_{q|k_h}{{\chi}(h_k)}Z^\prime_{0,0,0,0,{h_k},\frac{k_h}q}(0,\overline{\chi})\end{split}\end{equation}
with
\begin{equation}\label{c_0^'}c^\prime_0(h,k,T)\ll \delta(h_k)\delta(k_h)(\log\log h_kk_h)^2.\end{equation}
The $Z^\prime$ terms may be written as  
\begin{equation}\begin{split}\label{Z^prime}Z^\prime_{0,0,0,0,m,n}(0,\chi)=
\overline{G(\chi)}\frac{L(1,\chi)^4}{L(2,\chi^2)}\delta^\prime(m)\delta^\prime(n)
\end{split}\end{equation}
where
\begin{equation}\delta^\prime(m)=\prod_{\substack{p|m\\p \text{ split}}}\left(1+m_p\frac{p-1}{p+1}\right)\prod_{\substack{p|m\\p \text{ inert}}}\left(1+m_p\frac{p+1}{p-1}\right).
\end{equation}     
\end{prop}

\subsection{Evaluating the main term} 

\begin{prop}\label{main term prop}Let $c_2(h,k,T)$ be given by (\ref{c_2}) and let $\alpha(n)$ be defined as in Lemma \ref{alpha lem}. Suppose  $X,T\to\infty$ with $X\ll (\log T)^{2-\epsilon}$. Then 
\begin{equation}\sum_{\substack{h,k\leq T^{\theta}\\h,k\in\mathcal{W}(X)}}\frac{\alpha(h)\alpha(k)c_2(h,k,T)}{hk}(h,k)=(1+o(1))\frac{\log T\cdot\log qT}{\left(e^\gamma \log X\right)^2}.
\end{equation}
\end{prop}
\begin{proof}
Inputting the formula for $c_2(h,k,T)$ we see that we are required to show 
\begin{equation}\begin{split}S_0:=&\sum_{\substack{h,k\leq T^{\theta}\\h,k\in\mathcal{W}(X)}}\frac{\alpha(h)\alpha(k)\delta(h_k)\delta(k_h)}{hk}(h,k)\Big[\log T\cdot\log qT+O(\log T\log h_kk_h)\Big]
\\=&(1+o(1))\frac{\pi^2}{6}L(1,\chi)^{-2}\prod_{p|d_\mathbb{K}}\left(1+\frac{1}{p}\right)
\frac{\log T\cdot\log qT}{\left(e^\gamma \log X\right)^2}.
\end{split} 
\end{equation}
We first group together the terms for which $(h,k)=g$. Replacing $h$ by $hg$ and $k$ by $kg$ we obtain
\begin{multline}\label{S_0}S_0=\sum_{\substack{g\leq Y\\g\in\mathcal{W}(X)}}\frac{1}{g}\sum_{\substack{k\leq Y/g\\k\in\mathcal{W}(X)}}\frac{\alpha(kg)\delta(k)}{k}
\sum_{\substack{h\leq Y/g\\h\in\mathcal{W}(X)\\(h,k)=1}}\frac{\alpha(hg)\delta(h)}{h}
\\\times\Big[\log T\cdot\log qT+O(\log T\log hk)\Big]
\end{multline}
where $Y=T^{\theta}$. Let us first estimate the error term. We have
\begin{equation}\begin{split}&\sum_{\substack{g\leq Y\\g\in\mathcal{W}(X)}}\frac{1}{g}\sum_{\substack{k\leq Y/g\\k\in\mathcal{W}(X)}}\frac{\alpha(kg)\delta(k)}{k}\sum_{\substack{h\leq Y/g\\h\in\mathcal{W}(X)\\(h,k)=1}}\frac{\alpha(hg)\delta(h)}{h}\log\left(hk\right)
\\\ll&\sum_{\substack{g\in\mathcal{L}(X)}}\frac{d(g)^2}{g}\sum_{\substack{h,k\in\mathcal{L}(X)}}\frac{d(k)^2d(h)^2}{hk}\log{hk}
\\\ll&\sum_{\substack{g\in\mathcal{L}(X)}}\frac{d(g)^2}{g}\Bigg(\sum_{\substack{m\in\mathcal{L}(X)}}\frac{d(m)^2\log{m}}{m}\Bigg)^2.
\end{split}
\end{equation}
Writing $f(\sigma)=\sum_{m\in\mathcal{L}(X)}d(m)^2m^{-\sigma}$ the inner sum is $-f^{\prime}(1)$. Since $f(\sigma)=\prod_{p\leq X}(1-p^{-\sigma})^{-4}(1-p^{-2\sigma})$ we see $f^{\prime}(1)\ll f(1)\sum_{p\leq X}\log p/(p-1)\ll \log^5 X$ and hence the above sum is $\ll\log^{14} X$. We can now turn to the main term and consider
\begin{equation}\label{S first}S:=\sum_{\substack{g\leq Y\\g\in\mathcal{W}(X)}}\frac{1}{g}\sum_{\substack{k\leq Y/g\\k\in\mathcal{W}(X)}}\frac{\alpha(kg)\delta(k)}{k}\sum_{\substack{h\leq Y/g\\h\in\mathcal{W}(X)\\(h,k)=1}}\frac{\alpha(hg)\delta(h)}{h}.
\end{equation}
We define the function $\mu^\prime:\mathrm{Im}(\mathfrak{N})\to\mathbb{C}$, $\mathfrak{N}(\mathfrak{a})\mapsto \mu(\mathfrak{a})$ where $\mu$ is the extension of the usual m\"obius function to ideals given by
\begin{equation}\mu(\mathfrak{a})=
\begin{cases}1&\text{\,if\,\,\,}\mathfrak{a}=\mathcal{O}_\mathbb{K},
\\(-1)^r&\text{\,if\,\,\,}\mathfrak{a}=\mathfrak{p}_1\mathfrak{p}_2\ldots\mathfrak{p}_r,
\\ 0 &\text{\,otherwise.}
\end{cases}
\end{equation}
So basically; for split and ramified primes $\mu^\prime(p)=-1$ and $\mu^\prime(p^j)=0$ for $j\geq 2$; for inert primes   $\mu(p^2)=-1$ and $\mu(p^{2j})=0$ for $j\geq 2$, and $\mu^\prime$ is multiplicative.   Similarly to the usual m\"obius function we now have 
\begin{equation}\sum_{\substack{d|h\\d|k\\d\in\mathrm{Im}(\mathfrak{N})}}\mu^\prime(d)=\begin{cases}1&\text{\,if\,\,\,}(h,k)=1
\\0&\text{otherwise}
\end{cases}
\end{equation}
 for $h,k\in\mathrm{Im}(\mathfrak{N})$. Substituting this into the sum over $h$ in $S$ we see
\begin{equation}\label{S mu}
\begin{split}
S=&\sum_{\substack{g\leq Y\\g\in\mathcal{W}(X)}}\frac{1}{g}\sum_{\substack{k\leq Y/g\\k\in\mathcal{W}(X)}}\frac{\alpha(kg)\delta(k)}{k}\sum_{\substack{h\leq Y/g\\h\in\mathcal{W}(X)}}\bigg(\sum_{\substack{d|h\\d|k\\d\in\mathrm{Im}(\mathfrak{N})}}\mu^\prime(d)\bigg)\frac{\alpha(hg)\delta(h)}{h}
\\=&\sum_{\substack{g\leq Y\\g\in\mathcal{W}(X)}}\frac{1}{g}\sum_{\substack{l\leq Y/g\\l\in\mathcal{W}(X)}}\frac{\mu^\prime(l)}{l^2}\Bigg(\sum_{\substack{m\leq Y/gl\\m\in\mathcal{W}(X)}}\frac{\alpha(glm)\delta(lm)}{m}\Bigg)^2.
\end{split}
\end{equation}
Manipulating the sums in this way allows us to avoid the rather technical and lengthy calculations involved in \cite{hybrid}.

We wish to extend these sums over all $\mathcal{W}(X)$ and this requires some estimates. These will follow in a similar fashion to that found between (\ref{P trunc}) and (\ref{P trunc 1}). Throughout we use $\alpha(m),\delta(m)\ll d(m)$ and $d(mn)\leq d(m)d(n)$. First, let $b$ be positive and small, then 
\begin{equation}
\begin{split}
\frac{1}{d(g)d(l)^2}\sum_{\substack{m> Y/lg\\m\in\mathcal{W}(X)}}\frac{\alpha(glm)\delta(lm)}{m}\ll&
\sum_{\substack{m> Y/lg\\m\in\mathcal{W}(X)}}\frac{d(m)^2}{m}
\ll\left(\frac{Y}{lg}\right)^{-b}\sum_{\substack{m\in\mathcal{W}(X)}}\frac{d(m)^2}{m^{1-b}}
\\\ll&\left(\frac{Y}{lg}\right)^{-b}\prod_{p\leq X}\left(1-p^{-1+b}\right)^{-4}\left(1-p^{-2(1-b)}\right)
\\\ll&\left(\frac{Y}{lg}\right)^{-b}e^{8X^b/\log X}\ll (lg)^bT^{-b\theta/2}.
\end{split}
\end{equation} 
 Second,
\begin{equation}\sum_{\substack{m\in\mathcal{W}(X)}}\frac{\alpha(glm)\delta(lm)}{m}\ll d(g)d(l)^2\sum_{\substack{m\in\mathcal{W}(X)}}\frac{d(m)^2}{m}\ll d(g)d(l)^2\log^4 X.
\end{equation}
From these it follows that the square of the sum over $m$ in (\ref{S mu}) is 
\begin{equation}\Bigg(\sum_{m\in\mathcal{W}(X)}\frac{\alpha(glm)\delta(lm)}{m}\Bigg)^2+O\left(d(g)^{2}d(l)^{4}(lg)^{2b} T^{-b\theta/4}\right).
\end{equation}
Similarly we find 
\begin{equation}\sum_{\substack{l\in\mathcal{W}(X)}}\frac{\mu^\prime(l)d(l)^4}{l^{2-2b}}\ll 1,\,\,\,\,\,\,\,\sum_{\substack{l>Y/g \\l\in\mathcal{W}(X)}}\frac{\mu^\prime(l)d(l)^4}{l^{2-2b}}\ll g^c T^{-c\theta},
\end{equation}
for some small $c>0$, and
\begin{equation}\sum_{\substack{g\in\mathcal{W}(X)}}\frac{d(g)^2}{g^{1-2b-c}}\ll T^\epsilon,\,\,\,\,\,\sum_{\substack{g>Y\\g\in\mathcal{W}(X)}}\frac{d(g)^2}{g^{1-2b-c}}\ll T^{-d\theta}
\end{equation}
for some small $d>0$. The above estimates give 
\begin{equation}
\begin{split}
S=&\bigg(\sum_{g\in\mathcal{W}(X)}-\sum_{\substack{g>Y\\g\in\mathcal{W}(X)}}\bigg)\frac{1}{g}\bigg(\sum_{l\in\mathcal{W}(X)}-\sum_{\substack{l>Y/g\\l\in\mathcal{W}(X)}}\bigg)\frac{\mu^\prime(l)}{l^2}
\\&\times \bigg[\Bigg(\sum_{m\in\mathcal{W}(X)}\frac{\alpha(glm)\delta(lm)}{m}\Bigg)^2+O\left(d(g)^{2}d(l)^{4}(lg)^{2b} T^{-b\theta/4}\right)\bigg]
\\=&(1+o(1))\sum_{\substack{g\in\mathcal{W}(X)}}\frac{1}{g}\sum_{\substack{l\in\mathcal{W}(X)}}\frac{\mu^\prime(l)}{l^2}\Bigg(\sum_{\substack{m\in\mathcal{W}(X)}}\frac{\alpha(glm)\delta(lm)}{m}\Bigg)^2.
\end{split}
\end{equation}
 Now,  since all coefficients in $S$ are multiplicative we may expand the sum into an Euler product:
\begin{equation}
\begin{split}\label{S prod}S
=&(1+o(1))\prod_{\substack{p\leq X\\p\mathrm{\,\,split}}}G(p)\prod_{\substack{p\leq \sqrt{X}\\p\mathrm{\,\,inert}}}G(p^2)\prod_{\substack{p\leq X\\p\mathrm{\,\,ramified}}}G(p)
\end{split}
\end{equation} 
with
\begin{equation}G(p)=\sum_{i,j,u,v\geq 0}\frac{\mu^\prime(p^j)\alpha(p^{i+j+u})\alpha(p^{i+j+v})\delta(p^{j+u})\delta(p^{j+v})}{p^{i+2j+u+v}}.
\end{equation}
Performing the various sums whilst using the support conditions of $\alpha$ and $\mu^\prime$ we see
\begin{equation}
\begin{split}
G(p)=&1+\frac{2\alpha(p)\delta(p)+\alpha(p)^2}{p}+\frac{2\alpha(p^2)\delta(p^2)+\alpha(p^2)^2+2\alpha(p)\alpha(p^2)\delta(p)}{p^2}.
\end{split}
\end{equation}

Recall that for a split prime $p$ we have $\delta(p^r)=1+r(p-1)/(p+1)$ and hence $\delta(p)=2p/(p+1)$ and $\delta(p^2)=2\delta(p)-1$. We also have $\alpha(p)=-2$ for all $p\leq X$, $\alpha(p^2)=1$ for $p\leq \sqrt{X}$ and $\alpha(p^2)=2$ for $\sqrt{X}<p\leq X$. A straightforward calculation now gives
\begin{equation}\label{s prod}
\begin{split}\prod_{\substack{p\leq X\\p\mathrm{\,\,split}}}G(p)=&\prod_{\substack{p\leq \sqrt{X}\\p\mathrm{\,\,split}}}\bigg(\frac{(1-1/p)^4}{1-1/p^2}\bigg)\prod_{\substack{\sqrt{X}<p\leq X\\p\mathrm{\,\,split}}}\bigg(\frac{(1-1/p)^4}{1-1/p^2}+O\left(\frac{1}{p^2}\right)\bigg)
\\=&\prod_{\substack{p\leq X\\p\mathrm{\,\,split}}}\bigg(\frac{(1-1/p)^4}{1-1/p^2}\bigg)\prod_{\substack{\sqrt{X}<p\leq X\\p\mathrm{\,\,split}}}\bigg(1+O\left(\frac{1}{p^2}\right)\bigg)
\\=&(1+o(1))\prod_{\substack{p\leq X\\p\mathrm{\,\,split}}}\bigg(1-\frac{1}{p}\bigg)^4\prod_{p\mathrm{\,\,split}}\bigg(1-\frac{1}{p^2}\bigg)^{-1}.
\end{split}
\end{equation}
In evaluating the remaining products in (\ref{S prod}) we note that $\alpha$ behaves the same on square inert primes as it does on ramified primes. The same goes for $\delta$ since the number 1 varies little. We describe the ramified case since the inert case is simply handled by replacing $p$ with $p^2$. 

For a ramified prime $p$ we have $\delta(p)=\delta(p^2)=1$, $\alpha(p)=-1$ for all $p\leq X$, $\alpha(p^2)=0$ for $p\leq\sqrt{X}$ and $\alpha(p^2)=1/2$ for $\sqrt{X}<p\leq X$. With this information we see
\begin{equation}\label{r prod}
\begin{split}\prod_{\substack{p\leq X\\p\mathrm{\,\,ramified}}}G(p)=&\prod_{\substack{p\leq \sqrt{X}\\p\mathrm{\,\,ramified}}}\bigg(1-\frac{1}{p}\bigg)\prod_{\substack{\sqrt{X}<p\leq X\\p\mathrm{\,\,ramified}}}\bigg(1-\frac{1}{p}+\frac{1}{4p^2}\bigg)
\\=&\prod_{\substack{p\leq X\\p\mathrm{\,\,ramified}}}\bigg(1-\frac{1}{p}\bigg)\prod_{\substack{\sqrt{X}<p\leq X\\p\mathrm{\,\,ramified}}}\bigg(1+O\left(\frac{1}{p^2}\right)\bigg)
\\=&(1+o(1))\prod_{\substack{p\leq X\\p\mathrm{\,\,ramified}}}\bigg(1-\frac{1}{p}\bigg).
\\=&(1+o(1))\prod_{\substack{p\leq X\\p\mathrm{\,\,ramified}}}\bigg(1-\frac{1}{p}\bigg)^2\prod_{p\mathrm{\,\,ramified}}\left(1+\frac{1}{p}\right)\left(1-\frac{1}{p^2}\right)^{-1}.
\end{split}
\end{equation}
 In the last line here we have used the fact that a prime is ramified if and only if it divides $d_\mathbb{K}$ and hence $\sum_{X<p|d_\mathbb{K}}1/p=o(1)$. Similarly, for inert primes we find 
\begin{equation}\label{i prod}
\begin{split}\prod_{\substack{p\leq \sqrt{X}\\p\mathrm{\,\,inert}}}G(p^2)
=&(1+o(1))\prod_{\substack{p\leq \sqrt{X}\\p\mathrm{\,\,inert}}}\bigg(1-\frac{1}{p^2}\bigg)
\\=&(1+o(1))\prod_{\substack{p\leq \sqrt{X}\\p\mathrm{\,\,inert}}}\bigg(1-\frac{1}{p^2}\bigg)^2 \prod_{\substack{p\mathrm{\,\,inert}}}\bigg(1-\frac{1}{p^2}\bigg)^{-1}
\end{split}
\end{equation}
Collecting the infinite products in (\ref{s prod}), (\ref{r prod}) and (\ref{i prod}) we acquire the factor 
\begin{equation}\frac{\pi^2}{6}\prod_{p|d_\mathbb{K}}\left(1+\frac{1}{p}\right).
\end{equation}
The remaining terms are then given by 
\begin{equation}(1+o(1))\prod_{\mathfrak{N}(\mathfrak{p})\leq X}\left(1-\frac{1}{\mathfrak{N}(\mathfrak{p})}\right)^2=(1+o(1))(L(1,\chi)e^\gamma\log X)^{-2}.
\end{equation}
\end{proof}

\subsection{Estimating the lower order terms}
By virtue of the upper bounds (\ref{c_1}), (\ref{c_0^'}) and Proposition \ref{main term prop} we are only required to evaluate the sum of the `big O' and $Z^\prime$ terms of formula (\ref{ded twist}). For the `big O' term we have
\begin{equation}
\begin{split}
&T^{-\frac{1}{4}+\epsilon}\sum_{\substack{h,k\leq T^{\theta}\\h,k\in\mathcal{W}(X)}}\frac{\alpha(h)\alpha(k)(h,k)}{hk}\left(\frac{hk}{(h,k)^2}\right)^{7/8+\epsilon}
\\\ll& T^{-\frac{1}{4}+\epsilon}\bigg(\sum_{n\leq T^{\theta}}{d(n)}\bigg)^2\ll T^{2\theta-\frac{1}{4}+\epsilon}
\end{split}
\end{equation} 
and so taking $\theta<1/12$ the error term is $o(1)$. 

We now estimate the sums involving the $Z^\prime$ terms. By (\ref{c_0}) and (\ref{Z^prime}) we see that we must consider sums of the form
\begin{equation}\begin{split}S^\prime:=&\sum_{\substack{h,k\leq Y\\h,k\in\mathcal{W}(X)}}\frac{\alpha(h)\alpha(k)}{hk}(h,k){\bf 1}_{q|h_k}\chi(k_h)\delta^\prime(h_k/q)\delta^\prime(k_h)
\\=&\sum_{\substack{g\leq Y\\g\in\mathcal{W}(X)}}\frac{1}{g}\sum_{\substack{k\leq Y/g\\k\in\mathcal{W}(X)}}\frac{\chi(k)\alpha(kg)\delta^\prime(k)}{k}\sum_{\substack{h\leq Y/g\\h\in\mathcal{W}(X)\\(h,k)=1}}{\bf 1}_{q|h}\frac{\alpha(hg)\delta^\prime(h/q)}{h}.
\end{split}
\end{equation}  
where $Y=T^\theta$.
The innermost sum is given by
\begin{equation}
\sum_{\substack{h\leq Y/qg\\qh\in\mathcal{W}(X)\\(qh,k)=1}}\frac{\alpha(qhg)\delta^\prime(h)}{qh}
\ll \sum_{\substack{h\leq Y/g\\h\in\mathcal{W}(X)\\(h,k)=1}}\frac{\alpha(hg)\delta^\prime(h)}{h}
\end{equation}
where we have used $|\alpha(qm)|\leq\alpha(m)$ which follows from  (\ref{q}) and the definition of $\alpha$. We deduce that $S^\prime$ is $\ll$ a sum of the form (\ref{S first}) with $\delta$ replaced by $\delta^\prime$. Using the bound $\delta^\prime(n)\leq d(n^2)$ we may follow the analysis of Proposition \ref{main term prop} to see that 
\begin{equation}S^\prime\ll(1+o(1))\prod_{\substack{p\leq X\\p\mathrm{\,\,split}}}G^\prime(p)\prod_{\substack{p\leq \sqrt{X}\\p\mathrm{\,\,inert}}}G^\prime(p^2)\prod_{\substack{p\leq X\\p\mathrm{\,\,ramified}}}G^\prime(p)
\end{equation}
where
\begin{equation}G^\prime(p)=1+\frac{2\alpha(p)\delta^\prime(p)+\alpha(p)^2}{p}+\frac{2\alpha(p^2)\delta^\prime(p^2)+\alpha(p^2)^2+2\alpha(p)\alpha(p^2)\delta^\prime(p)}{p^2}.
\end{equation}
For split and ramified primes we have $\delta^\prime(p^r)=\delta(p^r)$ and so we only need evaluate $G$ at the inert primes. For inert $p$ we have $\delta(p^2)=1+2(p+1)/(p-1)\leq 5$ and hence 
\begin{equation}G^\prime(p^2)=1+O\left(\frac{1}{p^2}\right)
\end{equation}
For the sake of argument we write
\begin{equation}\prod_{\substack{p\leq \sqrt{X}\\p\mathrm{\,\,inert}}}G^\prime(p^2)=(1+o(1))\prod_{\substack{p\leq \sqrt{X}\\p\mathrm{\,\,inert}}}\left(1-\frac{1}{p^2}\right)^{2}
\end{equation}
and then combine this with the products over split and ramified primes given by (\ref{s prod}) and (\ref{r prod}). This gives
\begin{equation}S^\prime\ll (\log X)^{-2}.
\end{equation}

\section{Conjecture \ref{quad conj} via the moments recipe}\label{recipe sec}

In this section we modify the recipe given in \cite{cfkrs} to reproduce Conjecture \ref{quad conj}. The recipe in question is concerned with primitive $L$-functions, so cannot be applied directly to our situation without some modification. Our modifications are based on Theorem \ref{twisted thm} and are in keeping with the reasoning of the original recipe. Let us first describe the process as it appears in \cite{cfkrs} with the Riemann zeta function as the example. 

Consider the shifted product 
\begin{equation}Z(s,\boldsymbol{\alpha})=\zeta(s+\alpha_1)\cdots\zeta(s+\alpha_k)\zeta(1-s-\alpha_{k+1})\cdots\zeta(1-s-\alpha_{2k})
\end{equation}
We first replace each occurrence of $\zeta$ with its approximate functional equation
\begin{equation}\zeta(s)=\sum_m\frac{1}{m^s}+\varkappa(s)\sum_n\frac{1}{n^{1-s}}
\end{equation}
and multiply out the the expression to give a sum of $2^{2k}$ terms. We throw away any terms that do not have an equal amount of $\varkappa(s+\alpha_i)$ and $\varkappa(1-s-\alpha_j)$ factors, the reason being that these terms are oscillatory. Indeed, by Stirling's formula we have
\begin{multline}\varkappa(s+\beta_1)\cdots\varkappa(s+\beta_J)\varkappa(1-s-\gamma_1)\cdots\varkappa(1-s-\gamma_K)
\\\sim\left(\frac{t}{2\pi e}\right)^{-i(J-K)}e^{i(J-K)\pi/4}\left(\frac{t}{2\pi}\right)^{-\sum \beta_j+\sum\gamma_k}
\end{multline}
which is oscillating unless $J=K$.
In each of the remaining $\binom{2k}{k}$ terms we retain only the diagonal from the sum, which we then extend over all positive integers.  If we denote the resulting expression by $M(s,\boldsymbol{\alpha})$ then the conjecture is 
\begin{equation}\int_{-\infty}^{\infty} Z\left(\frac{1}{2}+it,\boldsymbol{\alpha}\right)w(t)dt\sim\int_{-\infty}^{\infty}M\left(\frac{1}{2}+it,\boldsymbol{\alpha}\right) w(t)dt
\end{equation}   
for any reasonable function $w(t)$. 

To describe a typical term of $M(s,\boldsymbol{\alpha})$ let us first define the prototypical diagonal sum 
\begin{equation}R(\sigma,\alpha_1,\ldots,\alpha_{2k})=\sum_{m_1\cdots m_k=n_1\cdots n_k}\frac{1}{m_1^{\sigma+\alpha_1}\cdots m_k^{\sigma+\alpha_k}n_1^{1-\sigma-\alpha_{k+1}}\cdots n_1^{1-\sigma-\alpha_{2k}}}.
\end{equation}
This is in fact the term acquired by taking the first sum of the approximate functional equation in each $\zeta$-factor of $Z(s,\boldsymbol{\alpha})$. If, for example, we were to take the second sum in $\zeta(s+\alpha_1)$ and the second sum in $\zeta(1-s-\alpha_{k+1})$ whilst taking the first in the rest we would acquire the term
\begin{equation}\left(\frac{t}{2\pi}\right)^{-\alpha_1+\alpha_{k+1}}R(\sigma,\alpha_{k+1},\alpha_2,\ldots,\alpha_{k},\alpha_1,\alpha_{k+2},\ldots,\alpha_{2k}).
\end{equation}
It is then clear that the full expression will be a sum over permutations $\tau\in S_{2k}$, and that any permutation other than the identity will swap elements of $\{\alpha_1,\ldots,\alpha_k\}$ with elements of $\{\alpha_{k+1},\ldots,\alpha_{2k}\}$ in the $R$ terms. Since $R$ is symmetric in the first $k$ variables and also in the second, we may reorder the entries such that the subscripts of the first $k$ are in increasing order, as are the last $k$. We thus see that the full expression will be a sum over the $\binom{2k}{k}$ permutations $\tau\in S_{2k}$ such that 
\begin{equation}\tau(1)<\ldots<\tau(k),\,\,\,\tau(k+1)<\ldots<\tau(2k).
\end{equation}  
Denote the set of such permutations by $\Xi$. A typical term now takes the form 
\begin{equation}\left(\frac{t}{2\pi}\right)^{(-\alpha_1-\cdots-\alpha_k+\alpha_{k+1}+\cdots+\alpha_{2k})/2}W(s,\boldsymbol{\alpha},\tau)
\end{equation}
with $\tau\in\Xi$ and where
\begin{equation}
\begin{split}W(s,\boldsymbol{\alpha},\tau)=&\left(\frac{t}{2\pi}\right)^{(\alpha_{\tau(1)}+\cdots+\alpha_{\tau(k)}-\alpha_{\tau(k+1)}-\cdots-\alpha_{\tau(2k)})/2}
\\&{\times R(\sigma,\alpha_{\tau(1)},\ldots,\alpha_{\tau(k)},\alpha_{\tau(k+1)},\ldots,\alpha_{\tau(2k)})}.
\end{split}
\end{equation}
Combining all terms we have
\begin{equation}\label{M}M(s,\boldsymbol{\alpha})=\left(\frac{t}{2\pi}\right)^{(-\alpha_1-\cdots-\alpha_k+\alpha_{k+1}+\cdots+\alpha_{2k})/2}\sum_{\tau\in\Xi}W(s,\boldsymbol{\alpha},\tau).
\end{equation}
To recover the $k$th-moment conjecture for the Riemann zeta function we first extract the polar behaviour of $R$. This gives 
\begin{equation}\label{R prod}R(\sigma,\alpha_1,\ldots,\alpha_{2k})=A_k(\sigma,\alpha_1,\ldots,\alpha_{2k})\prod_{i,j=1}^k\zeta(1+\alpha_i-\alpha_{k+j})
\end{equation}
where $A_k$ is some Euler product that is absolutely convergent for $\sigma>1/4$. Now, in \cite{cfkrs} it is shown (Lemma 2.5.1) that the sum over permutations in (\ref{M}) can be written as a contour integral. We reproduce this here as we shall have use for it later. 

\begin{lem}[\cite{cfkrs}]\label{2.5.1}Suppose $F(a;b)=F(a_1,\ldots,a_k;b_1,\ldots,b_k)$ is a function of $2k$ variables which is symmetric with respect to the first $k$ and also symmetric with respect to the second set of $k$ variables. Suppose also that $F$ is regular near $(0,\ldots,0)$. Suppose further that $f(s)$ has a simple pole of residue 1 at $s=0$ but is otherwise analytic in a neighbourhood about $s=0$. Let
\begin{equation}K(a_1,\ldots,a_k;b_1,\ldots,b_k)=F(a_1,\ldots,a_k;b_1,\ldots,b_k)\prod_{i,j=1}^k f(a_i-b_j).
\end{equation}
If for all $1\leq i,j\leq k$, $\alpha_i-\alpha_{k+j}$ is contained in the region of analyticity of $f(s)$ then 
\begin{equation}
\begin{split}&\sum_{\tau\in\Xi}K(\alpha_{\tau(1)},\ldots,\alpha_{\tau(k)};\alpha_{\tau(k+1)},\ldots,\alpha_{\tau(2k)})
\\=&\frac{(-1)^k}{k!^2(2\pi i)^{2k}}\oint\cdots\oint K(z_1,\ldots,z_k,z_{k+1},\ldots,z_{2k})
\frac{\Delta^2(z_1,\ldots,z_{2k})}{\prod_{i,j=1}^{k}(z_i-\alpha_j)}dz_1\cdots dz_{2k}
\end{split}
\end{equation}
\end{lem}

By the above Lemma and (\ref{M}), (\ref{R prod}) we see that $M(1/2+it,\boldsymbol{0})$ is given by
\begin{equation}
\begin{split}&\frac{(-1)^k}{k!^2(2\pi i)^{2k}}\oint\cdots\oint A_k(1/2,z_1,\ldots,z_{2k})\prod_{i,j=1}^k\zeta(1+z_i-z_{k+j})
\\&\times\frac{\Delta^2(z_1,\ldots,z_{2k})}{\prod_{j=1}^{2k}z_j^{2k}}\exp\bigg(\frac{1}{2}\log(t/2\pi)\sum_{j=1}^kz_j-z_{k+j}\bigg)dz_1\cdots dz_{2k}
\\=&A_k(1/2,0,\ldots,0)\log^{k^2}\left(\frac{t}{2\pi}\right)(1+O((\log t)^{-1}))\frac{(-1)^k}{2^{k^2}k!^2(2\pi i)^{2k}}
\\&\times\oint\cdots\oint \frac{\Delta^2(z_1,\ldots,z_{2k})}{\left(\prod_{i,j=1}(z_i-z_{k+j})\right)\prod_{j=1}^{2k}z_j^{2k}}e^{\sum_{j=1}^kz_j-z_{k+j}}dz_1\cdots dz_{2k}
\end{split}
\end{equation} 
after a change of variables. It can then be shown that $A_k(1/2,0,\ldots,0)=a(k)$, where $a(k)$ is given by (\ref{zeta arith}), and that the remaining terms give $G(k+1)^2/G(2k+1)$. 

We now turn our attention to the shifted product
\begin{equation}
\begin{split}Z(s,\boldsymbol{\alpha},\boldsymbol{\beta})=Z_\zeta(s,\boldsymbol{\alpha})Z_L(s,\boldsymbol{\beta})
 \end{split}
\end{equation}
where
 \begin{equation}Z_\zeta(s,\boldsymbol{\alpha})=\zeta(s+\alpha_1)\cdots\zeta(s+\alpha_k)\zeta(1-s-\alpha_{k+1})\cdots\zeta(1-s-\alpha_{2k})
\end{equation}
and
\begin{equation} Z_L(s,\boldsymbol{\beta})=L(s+\beta_1,\chi)\cdots L(s+\beta_k,\chi)L(1-s-\beta_{k+1},\overline{\chi})\cdots L(1-s-\beta_{2k},\overline{\chi}).
\end{equation}
As before, we plan to substitute the respective approximate functional equations, which we now write as
\begin{equation}\zeta(s)=\sum_m\frac{1}{m^s}+\varkappa_\zeta(s)\sum_n\frac{1}{n^{1-s}},
\end{equation}
\begin{equation}L(s,\chi)=\sum_m\frac{\chi(m)}{m^s}+\varkappa_L(s)\sum_n\frac{\overline{\chi}(n)}{n^{1-s}}.
\end{equation}
We have
\begin{equation}\varkappa_L(s)=\frac{G(\chi)}{i^\mathfrak{a}\sqrt{q}}\left(\frac{\pi}{q}\right)^{-\frac{1}{2}+s}\frac{\Gamma((1-s+\mathfrak{a})/2)}{\Gamma((s+\mathfrak{a})/2)}
\end{equation}
where $G(\chi)$ is the Gauss sum of $\chi$ and $\mathfrak{a}$ is defined by the equation $\chi(-1)=(-1)^\mathfrak{a}$.
An exercise with Stirling's formula gives
\begin{equation}\varkappa_L(s)=\frac{G(\chi)}{i^\mathfrak{a}\sqrt{q}}\left(\frac{qt}{2\pi}\right)^{\frac{1}{2}-s}e^{it+i\pi/4}\left(1+O\left(\frac{1}{t}\right)\right).
\end{equation}
If we now follow the recipe and treat the $L$-functions as if they were zeta functions, then after expanding  and throwing away the terms with an unequal amount of $\varkappa_{\zeta,L}(s+\gamma_i)$ and $\varkappa_{\zeta,L}(1-s-\delta_j)$ factors, we are still left with some terms that have a factor of $q^{-it}$. Since this is oscillating we modify the recipe to throw these terms away also. We note with this modification the recipe reproduces Theorem \ref{twisted thm}, which adds some justification. 

One way of arriving at the resultant expression is is to apply the first step of the recipe to $Z_\zeta(s,\boldsymbol{\alpha})$ and $Z_L(s,\boldsymbol{\beta})$ separately. This prevents the occurrence of the aforementioned terms. Here, we note that when applying this step to $Z_L$ one can use the fact that $G(\chi)G(\overline{\chi})=(-1)^\mathfrak{a}q$ to provide some cancellation. We now form the product to gain a sum of $\binom{2k}{k}^2$ terms and retain only the diagonals as before. Extending the sums over all positive integers we then denote the resulting expression by $M(s,\boldsymbol{\alpha},\boldsymbol{\beta})$ and conjecture that  
\begin{equation}\int_{-\infty}^{\infty} Z\left(\frac{1}{2}+it,\boldsymbol{\alpha},\boldsymbol{\beta}\right)w(t)dt\sim\int_{-\infty}^{\infty}M\left(\frac{1}{2}+it,\boldsymbol{\alpha},\boldsymbol{\beta}\right) w(t)dt.
\end{equation} 

Applying the modified recipe and using a similar reasoning given in the case of the Riemann zeta function above, we see that
\begin{equation}
\begin{split}\label{new M}M(s,\boldsymbol{\alpha},\boldsymbol{\beta})=&\left(\frac{t}{2\pi}\right)^{(-\alpha_1-\cdots-\alpha_k+\alpha_{k+1}+\cdots+\alpha_{2k})/2}\left(\frac{qt}{2\pi}\right)^{(-\beta_1-\cdots-\beta_k+\beta_{k+1}+\cdots+\beta_{2k})/2}
\\&\times\sum_{\tau,\tau^\prime\in\Xi}W(s,\boldsymbol{\alpha},\boldsymbol{\beta},\tau,\tau^\prime)
\end{split}
\end{equation} 
where
\begin{equation}
\begin{split}
W(s,\boldsymbol{\alpha},\boldsymbol{\beta},\tau,\tau^\prime)=&\left(\frac{t}{2\pi}\right)^{(\alpha_{\tau(1)}+\cdots+\alpha_{\tau(k)}-\alpha_{\tau(k+1)}-\cdots-\alpha_{\tau(2k)})/2}
\\&\times \left(\frac{qt}{2\pi}\right)^{(\beta_{\tau^\prime(1)}+\cdots+\beta_{\tau^\prime(k)}-\beta_{\tau^\prime(k+1)}-\cdots-\beta_{\tau^\prime(2k)})/2}
\\&{\times S(\sigma;\alpha_{\tau(1)},\ldots,\alpha_{\tau(2k)};\beta_{\tau^\prime(1)},\ldots,\beta_{\tau^\prime(2k)})}
\end{split}
\end{equation}
with
\begin{multline}S(\sigma;\alpha_1,\ldots,\alpha_{2k};\beta_1,\ldots,\beta_{2k})
\\=\sum_{\substack{m_1\cdots m_km^\prime_1\cdots m^\prime_k=\\n_1\cdots n_kn^\prime_1\cdots n^\prime_k}}{\chi(m_1^\prime)\cdots\chi(m_k^\prime)\overline{\chi(n_1^\prime)\cdots\chi(n_k^\prime)}}\bigg[m_1^{\sigma+\alpha_1}\cdots m_k^{\sigma+\alpha_k}
\\\times {m^\prime_1}^{\sigma+\beta_1}\cdots {m^\prime_k}^{\sigma+\beta_k}
n_1^{1-\sigma-\alpha_{k+1}}\cdots n_k^{1-\sigma-\alpha_{2k}}{n^\prime_1}^{1-\sigma-\beta_{k+1}}\cdots {n^\prime_k}^{1-\sigma-\beta_{2k}}\bigg]^{-1}.
\end{multline}
Since the condition $m_1\cdots m_km^\prime_1\cdots m^\prime_k=n_1\cdots n_kn^\prime_1\cdots n^\prime_k$ is multiplicative we have 
\begin{align*}
&S(\sigma;\alpha_1,\ldots,\alpha_{2k};\beta_1,\ldots,\beta_{2k})
\\=&{\prod_p\sum_{\substack{\sum_{j=1}^k e_j+e^\prime_j=\\\sum_{j=1}^k e_{j+k}+e^\prime_{j+k}\\}}{\chi(p^{e_1^\prime})\cdots\chi(p^{e_k^\prime})\overline{\chi(p^{e_{k+1}^\prime})\cdots\chi(p^{e_{2k}^\prime})}}\bigg[p^{e_1({\sigma+\alpha_1})}\cdots p^{e_k({\sigma+\alpha_k})}}
\\&\times { p^{e^\prime_1({\sigma+\beta_1})}\cdots p^{e^\prime_k({\sigma+\beta_k})}
p^{e_{k+1}({1-\sigma-\alpha_{k+1}})}\cdots p^{e_{2k}({1-\sigma-\alpha_{2k}})}}
\\&\times p^{e^\prime_{k+1}({1-\sigma-\beta_{k+1}})}\cdots p^{e_{2k}^\prime({1-\sigma-\beta_{2k}})}\bigg]^{-1}
\displaybreak\\=& A_k(\sigma,\boldsymbol{\alpha},\boldsymbol{\beta})\prod_{i,j=1}^k\zeta(1+\alpha_i-\alpha_{k+j})L(1+\beta_i-\beta_{k+j},|\chi|^2)
\\&\times L(1+\beta_i-\alpha_{j+k},\chi)L(1+\alpha_i-\beta_{j+k},\overline{\chi})
\\=& A_k(\sigma,\boldsymbol{\alpha},\boldsymbol{\beta})\prod_{i,j=1}^k\zeta(1+\alpha_i-\alpha_{k+j})\zeta(1+\beta_i-\beta_{k+j})
\\&\times L(1+\beta_i-\alpha_{j+k},\chi)L(1+\alpha_i-\beta_{j+k},\overline{\chi})\bigg(\prod_{p|q}\left(1-p^{-1-\beta_i+\beta_{k+j}}\right)\bigg)
\end{align*}
where $A_k$ is an Euler product that is absolutely convergent for $\sigma>1/4$. For $\sigma=1/2$ we have the following explicit expression for $A_k$:
\begin{equation}
\begin{split} A_k(1/2,\boldsymbol{\alpha},\boldsymbol{\beta})=&\prod_p\prod_{i,j=1}^k(1-p^{-1-\alpha_i+\alpha_{j+k}})
(1-|\chi(p)|^2p^{-1-\beta_i+\beta_{j+k}})
\\&\times (1-\chi(p)p^{-1-\beta_i+\alpha_{j+k}})(1-\overline{\chi}(p)p^{-1-\alpha_i+\beta_{j+k}})B_p(\boldsymbol{\alpha},\boldsymbol{\beta})
\end{split}
\end{equation}
where 
\begin{equation}
\begin{split}
B_p(\boldsymbol{\alpha},\boldsymbol{\beta})=&\sum_{\substack{\sum_{j=1}^k e_j+e^\prime_j=\\\sum_{j=1}^k e_{j+k}+e^\prime_{j+k}}}\frac{\chi(p^{e_1^\prime})\cdots\overline{\chi(p^{e_{2k}^\prime})}}{p^{e_1({1/2+\alpha_1})}\cdots p^{e_{2k}^\prime({1/2-\beta_{2k}})}}
\\=&\int_0^1\sum_{\substack{e_1,\ldots,e_{2k}\\e^\prime_1,\ldots,e^\prime_{2k}}}\frac{\chi(p^{e_1^\prime})\cdots\overline{\chi(p^{e_{2k}^\prime})}}{p^{e_1({1/2+\alpha_1})}\cdots p^{e_{2k}^\prime({1/2-\beta_{2k}})}}
\\&\times e\left(\left(\sum_{j=1}^k e_j+e^\prime_j-\sum_{j=1}^k e_{j+k}+e^\prime_{j+k}\right)\theta\right)d\theta
\\=&\int_0^1\prod_{j=1}^k\sum_{e_j=0}^\infty\frac{1}{p^{e_j(1/2+\alpha_j)}}e(e_j\theta)\prod_{j=1}^k\sum_{e_{j+k}=0}^\infty\frac{1}{p^{e_{j+k}(1/2-\alpha_{j+k})}}e(-e_{j+k}\theta)
\\&\times\prod_{j=1}^k\sum_{e^\prime_j=0}^\infty\frac{\chi(p^{e^\prime_j})}{p^{e^\prime_j(1/2+\beta_j)}}e(e^\prime_j\theta)\prod_{j=1}^k\sum_{e^\prime_{j+k}=0}^\infty\frac{\overline{\chi}(p^{e_{j+k}^\prime})}{p^{e^\prime_{j+k}(1/2-\beta_{j+k})}}e(-e^\prime_{j+k}\theta)d\theta
\\=&\int_0^1 \prod_{j=1}^k\zeta_p\left(\frac{e(\theta)}{p^{1/2+\alpha_j}}\right)\zeta_p\left(\frac{e(-\theta)}{p^{1/2\alpha_{k+j}}}\right)L_p\left(\frac{e(\theta)}{p^{1/2+\beta_j}}\right)\overline{L}_p\left(\frac{e(-\theta)}{p^{1/2-\beta_{j+k}}}\right)d\theta
\end{split}
\end{equation}
with $\zeta_p(x)=(1-x)^{-1}$, $L_p(x)=(1-\chi(p)x)^{-1}$ and $\overline{L}_p(x)=(1-\overline{\chi}(p)x)^{-1}$.

Now, denote the holomorphic part of $S(1/2,\boldsymbol{\alpha},\boldsymbol{\beta})$ by
\begin{equation}
\begin{split}A^\prime_k(1/2,\boldsymbol{\alpha},\boldsymbol{\beta})=&A_k(1/2,\boldsymbol{\alpha},\boldsymbol{\beta})\prod_{i,j=1}^k L(1+\beta_i-\alpha_{j+k},\chi)L(1+\alpha_i-\beta_{j+k},\overline{\chi})
\\&\times\bigg(\prod_{p|q}\left(1-p^{-1-\beta_i+\beta_{k+j}}\right)\bigg).
\end{split}
\end{equation}
Applying Lemma \ref{2.5.1} twice to (\ref{new M}) we see that  $M(1/2+it,\boldsymbol{0},\boldsymbol{0})$ is given by
\begin{equation}
\begin{split}&\left(\frac{(-1)^k}{k!^2(2\pi i)^{2k}}\right)^2\oint\cdots\oint A^\prime_k(1/2,u_1,\ldots,u_{2k},v_1,\ldots,v_{2k})
\\&\times\prod_{i,j=1}^k\zeta(1+u_i-u_{k+j})\zeta(1+v_i-v_{k+j})\frac{\Delta^2(u_1,\ldots,u_{2k})}{\prod_{j=1}^{2k}u_j^{2k}}\frac{\Delta^2(v_1,\ldots,v_{2k})}{\prod_{j=1}^{2k}v_j^{2k}}
\\&\times e^{\frac{1}{2}\mathcal{L}\sum_{j=1}^ku_j-u_{k+j}}e^{\frac{1}{2}\mathcal{L}_q\sum_{j=1}^kv_j-v_{k+j}}du_1\cdots du_{2k}dv_1\cdots dv_{2k}
\end{split}
\end{equation}
where $\mathcal{L}=\log (t/2\pi)$ and $\mathcal{L}_q=\log(qt/2\pi)$.
Since $A_k(1/2,\boldsymbol{\alpha},\boldsymbol{\beta})$ is holomorphic in the neighbourhood of $(\boldsymbol{\alpha},\boldsymbol{\beta})=(\boldsymbol{0},\boldsymbol{0})$ after a change of variables this becomes 
\begin{equation}
\begin{split}\label{contour int}&\left(\frac{(-1)^k}{k!^2(2\pi i)^{2k}}\right)^2\oint\cdots\oint A^\prime_k\left(1/2,\frac{u_1}{\mathcal{L}/2},\ldots,\frac{u_{2k}}{\mathcal{L}/2},\frac{v_1}{\mathcal{L}_q/2},\ldots,\frac{v_{2k}}{\mathcal{L}_q/2}\right)
\\&\times\prod_{i,j=1}^k\zeta\left(1+\frac{u_i-u_{k+j}}{\mathcal{L}/2}\right)\zeta\left(1+\frac{v_i-v_{k+j}}{\mathcal{L}_q/2}\right)\frac{\Delta^2(u_1,\ldots,u_{2k})}{\prod_{j=1}^{2k}u_j^{2k}}
\\&\times \frac{\Delta^2(v_1,\ldots,v_{2k})}{\prod_{j=1}^{2k}v_j^{2k}}e^{\sum_{j=1}^ku_j-u_{k+j}}e^{\sum_{j=1}^kv_j-v_{k+j}}du_1\cdots du_{2k}dv_1\cdots dv_{2k}.
\\=&A^\prime_k\left(1/2,\boldsymbol{0},\boldsymbol{0}\right)\mathcal{L}^{k^2}\mathcal{L}_q^{k^2}\left(1+O\left(\frac{1}{\mathcal{L}}\right)\right)\left(\frac{(-1)^k}{2^{k^2}k!^2(2\pi i)^{2k}}\right)^2\oint\cdots\oint 
\\&\times\frac{\Delta^2(u_1,\ldots,u_{2k})}{\prod_{i,j=1}^k(u_i-u_{k+j})\prod_{j=1}^{2k}u_j^{2k}}\frac{\Delta^2(v_1,\ldots,v_{2k})}{\prod_{i,j=1}^k(v_i-v_{k+j})\prod_{j=1}^{2k}v_j^{2k}}
\\&\times e^{\sum_{j=1}^ku_j-u_{k+j}}e^{\sum_{j=1}^kv_j-v_{k+j}}du_1\cdots du_{2k}dv_1\cdots dv_{2k}.
\\\sim&A^\prime_k\left(1/2,\boldsymbol{0},\boldsymbol{0}\right)\mathcal{L}^{k^2}\mathcal{L}_q^{k^2}\Bigg(\frac{(-1)^k}{2^{k^2}k!^2(2\pi i)^{2k}}\oint\cdots\oint 
\\&\times\frac{\Delta^2(u_1,\ldots,u_{2k})}{\prod_{i,j=1}^k(u_i-u_{k+j})\prod_{j=1}^{2k}u_j^{2k}}
e^{\sum_{j=1}^ku_j-u_{k+j}}du_1\cdots du_{2k}\Bigg)^2.
\end{split}
\end{equation}
As previously mentioned, it is shown in \cite{cfkrs} that the quantity in parentheses is given by $G(k+1)^2/G(2k+1)$ and so it only remains to show that $A^\prime_k\left(1/2,\boldsymbol{0},\boldsymbol{0}\right)=a(k)L(1,\chi)^{2k^2}$ where $a(k)$ is given by (\ref{arith factor}). Since, 
\begin{equation}A^\prime_k\left(1/2,\boldsymbol{0},\boldsymbol{0}\right)=A_k\left(1/2,\boldsymbol{0},\boldsymbol{0}\right)L(1,\chi)^{2k^2}\prod_{p|q}(1-p^{-1})^{k^2}
\end{equation}
we only need show that $a(k)$ is given by the quantity
\begin{equation}
\begin{split}&\prod_p\big[(1-p^{-1})
(1-|\chi(p)|^2p^{-1}) (1-\chi(p)p^{-1})(1-\overline{\chi}(p)p^{-1})\big]^{k^2}
\\&\times B_p(\boldsymbol{0},\boldsymbol{0})\prod_{p|q}(1-p^{-1})^{k^2}
\\=&\prod_p\big[(1-p^{-1})
 (1-\chi(p)p^{-1})\big]^{2k^2}
B_p(\boldsymbol{0},\boldsymbol{0})
\\=&b(k),
\end{split}
\end{equation}
say. In the case of quadratic extensions, $a(k)$ is the product of the following three factors
\begin{align}\label{prod over split}\prod_{\substack{p\mathrm{\,\,split}}}\left(1-\frac{1}{p}\right)^{4k^2}&\sum_{m=1}^\infty \frac{d_{2k}(p^m)^2}{p^m},
\\\label{prod over inert}\prod_{\substack{p\mathrm{\,\,inert}}}\left(1-\frac{1}{p^2}\right)^{2k^2}&\sum_{m=1}^\infty \frac{d_{k}(p^m)^2}{p^{2m}},
\\\label{prod over ram}\prod_{\substack{p\mathrm{\,\,ramified}}}\left(1-\frac{1}{p}\right)^{2k^2}&\sum_{m=1}^\infty \frac{d_{k}(p^m)^2}{p^m}.
\end{align}
Now, since $\chi(p)=1$ for split primes, the relevant factor in $b(k)$ is given by
\begin{equation}
\begin{split}\prod_{\substack{p\mathrm{\,\,split}}}\left(1-\frac{1}{p}\right)^{4k^2}\int_0^1 \left(1-\frac{e(\theta)}{p^{1/2}}\right)^{-2k}\left(1-\frac{e(-\theta)}{p^{1/2}}\right)^{-2k}d\theta.
\end{split}
\end{equation}
Since $k$ is an integer we can expand the integrand into a double series. Upon integration this is easily seen to be equal to the sum in (\ref{prod over split}) after using
\begin{equation}\binom{m+2k-1}{2k-1}=\binom{m+2k-1}{m}=d_{2k}(p^m).
\end{equation}
For inert primes we have $\chi(p)=-1$ and so the relevant factor is 
\begin{equation}
\begin{split}\prod_{\substack{p\mathrm{\,\,inert}}}\left(1-\frac{1}{p^2}\right)^{2k^2}\int_0^1 \left(1-\frac{e(2\theta)}{p}\right)^{-k}\left(1-\frac{e(-2\theta)}{p}\right)^{-k}d\theta,
\end{split}
\end{equation}
which is again easily seen to be equal to (\ref{prod over inert}). Finally, for ramified primes, or equivalently the primes dividing $q$, we have $\chi(p)=0$. Therefore, this factor in $b(k)$ is given by
\begin{equation}
\begin{split}\prod_{\substack{p\mathrm{\,\,ramified}}}\left(1-\frac{1}{p}\right)^{2k^2}\int_0^1 \left(1-\frac{e(\theta)}{p^{1/2}}\right)^{-k}\left(1-\frac{e(-\theta)}{p^{1/2}}\right)^{-k}d\theta
\end{split}
\end{equation}
which equals (\ref{prod over ram}).

\section{Moments of general non-primitive $L$-functions}\label{non-prim sec}

A key point in both derivations of Conjecture \ref{quad conj} was that, aside from the arithmetic factor, the leading term in the moment of $\zeta(1/2+it) L(1/2+it,\chi)$ was given by the product of the leading terms of the moments of $\zeta(1/2+it)$ and $L(1/2+it,\chi)$. We believe this should be the case for general non-primitive $L$-functions too. Indeed, by applying our modified moments recipe to non-primitive $L$-functions this idea becomes more apparent. 

The recipe for general non-primitive $L$-functions goes as follows. Suppose we have the product
$L(s)=\prod_{j=1}^m L_j(s)^{e_j}$ where the $L_j(s)$ are distinct, primitive members of the Selberg class $\mathcal{S}$. Suppose for each $L_j(s)$ we have the functional equation
\begin{equation}\xi_{L_j}(s)=\gamma_{L_j}(s)L_j(s)=\epsilon\overline{\xi}_{L_j}(1-s)
\end{equation}
where
\begin{equation}\gamma_{L_j}(s)=Q_j^{s/2}\prod_{i=1}^{d_j}\Gamma(s/2+\mu_{i,j})
\end{equation}
with the $\{\mu_{i,j}\}$ stable under complex conjugation. 
We then have the approximate functional equations
\begin{equation}L_j(s)=\sum_n \frac{\alpha_{L_j}(m)}{m^s}+\varkappa_{L_j}(s)\sum_n  \frac{\overline{\alpha_{L_j}(n)}}{n^s}
\end{equation}
where 
\begin{equation}\varkappa_{L_j}(s)=\frac{\overline{\gamma_{L_j}(1-s)}}{\gamma_{L_j}(s)}=Q_j^{1/2-s}\prod_{i=1}^{d_j}\frac{\Gamma((1-s)/2+\overline{\mu_{i,j}})}{\Gamma(s/2+\mu_{i,j})}.
\end{equation}
Similarly to before, if we apply the original recipe we encounter terms of the form $(Q_jQ_{j^\prime})^{-it}$ which are oscillating. We can prevent the occurrence of these terms by applying the first step of the recipe to each $L_j(s)$ separately. We then continue as in the original recipe.  It should be clear that when the resulting expression is written as a contour integral, the same manipulations used on (\ref{contour int}) will allow for a factorisation of the main term. 

In terms of the random matrix theory, let us assume that we have a hybrid product for $L(s)$. Since the $L_j(s)$ are distinct their zeros are uncorrelated \cite{liu ye}, and so their associated matrices should act independently. Hence, when the moment of the product over zeros is considered as an expectation, it will factorise. 

As we have already seen, this phenomenon occurs when considering $\zeta(s)L(s,\chi)$. Let us restate the conjecture in the more descriptive form 
\begin{multline}\frac{1}{T}\int_0^T\left|\zeta\left(\frac{1}{2}+it\right)^kL\left(\frac{1}{2}+it,\chi\right)^k\right|^{2}dt
\\\sim a(k)\frac{G(k+1)^2}{G(2k+1)}\log^{k^2}T\cdot\frac{G(k+1)^2}{G(2k+1)}\log^{k^2}qT,
\end{multline}  
with
\begin{equation}a(k)=\prod_p \left(1-\frac{1}{p}\right)^{2k^2}\sum_{m\geq 0}\frac{|F_{\chi,k}(p^m)|^2}{p^m},
\end{equation}
\begin{equation}F_{\chi,k}(n)=\sum_{n_1n_2=n}d_k(n_1)d_k(n_2)\chi(n_2).
\end{equation}
 The coefficients $F_{\chi,k}(n)$ are, of course, the Dirichlet coefficients of $\zeta(s)^kL(s,\chi)^k$.

As another example, we state a result to appear in a forthcoming joint paper between the author and Caroline Turnage-Butterbaugh. 
Here it is established, by an application of Theorem \ref{twisted thm}, that 
\begin{multline}\frac{1}{T}\int_0^T\left|\zeta\left(\frac{1}{2}+it\right)L\left(\frac{1}{2}+it,\chi\right)\sum_{n\leq T^\theta}\frac{1}{n^{1/2+it}}\right|^2dt
\\\sim b(1)\log^4 T\cdot\log qT \left(\frac{4\theta^3-3\theta^4}{12}\right)
\end{multline}  
where 
\begin{equation}b(1)=\prod_p \left(1-\frac{1}{p}\right)^{5}\sum_{m\geq 0}\frac{|H_\chi(p^m)|^2}{p^m},\,\,\,\,\,\,\,\,\,\,\,H_\chi(n)=\sum_{n_1n_2=n}d(n_1)\chi(n_2),
\end{equation}
and $\theta<1/11-\epsilon$. It is expected that Theorem \ref{twisted thm} remains valid for  $\theta=1$, in which case the above relation reads as 
\begin{multline}
\frac{1}{T}\int_0^T\left|\zeta\left(\frac{1}{2}+it\right)^2L\left(\frac{1}{2}+it,\chi\right)\right|^2dt
\\\sim \frac{b(1)}{12}\log^4 T\cdot\log qT
=b(1)\cdot\frac{G(3)^2}{G(5)}\log^4 T\cdot \frac{G(2)^2}{G(3)}\log qT.
\end{multline}
In terms of the $T$ behaviour, this can be thought of as the product of the fourth moment of zeta times the second moment of $L$.
Again, this is consistent with our random matrix theory/moments recipe reasoning.
Guided by these examples we are led to Conjecture \ref{non-prim conj} which, after ignoring the conductors, we restate as 
\begin{equation}\frac{1}{T}\int_0^T\left|L\left(\frac{1}{2}+it\right)\right|^{2k}dt\sim \frac{a_L(k)g_L(k)}{\Gamma(n_Lk^2+1)}\log^{n_Lk^2}T
\end{equation}
where $n_L=\sum_{j=1}^m e^2_j$,
\begin{equation}g_L(k)=\Gamma(n_Lk^2+1)\prod_{j=1}^m\frac{G^2(e_jk+1)}{G(2e_jk+1)}d_j^{(e_jk)^2},
\end{equation}
and
\begin{equation}a_L(k)=\prod_p \left(1-\frac{1}{p}\right)^{n_Lk^2}\sum_{n=0}^\infty \frac{|\alpha_{L,k}(p^n)|^2}{p^n}.
\end{equation}

Let us cast this in the light of some of the Selberg's conjectures. First,
we note that the integer $n_L$ is the same integer appearing in Selberg's `regularity of distribution' conjecture:
\begin{equation}\label{reg of dist}\sum_{p\leq x}\frac{|\alpha_L(p)|^2}{p}=n_L\log\log x+O(1).
\end{equation}
This is not so surprising since one expects the mean square of $L(1/2+it)$ to be asymptotic to a multiple of the sum $\sum_{n\leq T}|\alpha_L(n)|^2 n^{-1}$. The implication of (\ref{reg of dist}) is that this sum is in fact $\sim (a_L(1)/n_L!)\log^{n_L} T$.

For general $k$, we outline a verification of this last assertion. We assume the following two conjectures of Selberg \cite{selberg}: For primitive $F\in\mathcal{S}$ we have 
\begin{equation}\sum_{p\leq x}\frac{|\alpha_F(p)|^2}{p}=\log\log x+O(1),
\end{equation}
and for two distinct and primitive $F,G\in\mathcal{S}$ we have 
\begin{equation}\label{orth conj}\sum_{p\leq x}\frac{\alpha_{F}(p)\overline{\alpha_{G}(p)}}{p}=O(1).
\end{equation}
We also require that the functions
\begin{equation}M_j(s)=\sum_{n=1}^\infty \frac{|\alpha_{L_j}(n)|^2}{n^s}
\end{equation}
behave `reasonably', in particular, that they posses an analytic continuation. 

Now, given the factorisation $L(s)=\prod_{j=1}^m L_j(s)^{e_j}$ into primitive functions we have
\begin{equation}\label{alpha expansion}
\begin{split}
\sum_{p\leq x}\frac{|\alpha_{L,k}(p)|^2}{p}
=&\sum_{p\leq x}k^2\left(\sum_{j=1}^m e_j^2|\alpha_{L_j}(p)|^2+\sum_{i\neq j}e_ie_j\alpha_{L_i}(p)\overline{\alpha_{L_j}(p)}\right)p^{-1}
\\=&n_Lk^2\log\log x+O(1).
\end{split}
\end{equation}
If $M(s)=\sum |\alpha_{L,k}(n)|^2n^{-s}$, then the above equation implies a factorisation of the form
\begin{equation}M(s)=U_k(s)\prod_{j=1}^m M_j(s)^{(e_jk)^2}
\end{equation}
where $U_k(s)$ is some Euler product that is absolutely convergent for $\sigma>1/2$. Therefore, we may analytically continue $M(s)$ to $\sigma>1/2$. 
Also, by applying partial summation to (\ref{alpha expansion}) we see
\begin{equation}\label{sum over p}\sum_{p}\frac{|\alpha_{L,k}(p)|^2}{p^{s+1}}=n_Lk^2\int_2^\infty \frac{dx}{x^{s+1}\log x}+\cdots=-n_Lk^2\log s+\cdots,
\end{equation}
for small $\sigma>0$. If we write
\begin{equation}
\begin{split}
M(s+1)=& \prod_p\left(1+\frac{|\alpha_{L,k}(p)|^2}{p^{s+1}}+\frac{|\alpha_{L,k}(p^2)|^2}{p^{2(s+1)}}+\cdots\right)
\\=& \prod_p\left(\exp\left(\frac{|\alpha_{L,k}(p)|^2}{p^{s+1}}\right)+E_k(p,s)\right)
\\=& \exp\left(\sum_p\frac{|\alpha_{L,k}(p)|^2}{p^{s+1}}\right) \prod_p\left(1+F_k(p,s)\right),
\end{split}
\end{equation}
where $E_k(p,s)$ and $F_k(p,s)$ are both $\ll p^{-2(\sigma+1)+\epsilon}$, we see that $M(s+1)$ has a pole of order $n_Lk^2$ at $s=0$. It is shown in \cite{conrey ghosh 2} that on the assumption of Selberg's conjectures, if $F\in\mathcal{S}$ has a pole of order $m$ at $s=1$ then $\zeta(s)^m$ divides $F(s)$. Consequently, the residue of $M(s+1)$ at $s=0$ is given by $a_L(k)$. The usual argument involving Perron's formula now gives
\vspace{0.3cm}
\begin{equation}\sum_{n\leq T}\frac{|\alpha_{L,k}(n)|^2}{n}\sim \frac{a_L(k)}{(n_Lk^2)!}\log^{n_Lk^2} T.
\end{equation}


\newpage
\begin{thebibliography}{9}
\bibitem{bm}{R.~W.~Bruggeman, Y.~Motohashi}, \textit{Fourth Power Moment of Dedekind Zeta Functions of Real Quadratic Number Fields With Class Number One}, {Functiones et Approximatio} {\bf 29} (2001) 41--79.

\bibitem{bm 2}{R.~W.~Bruggeman, Y.~Motohashi}, \textit{Sum formula for Kloosterman sums and fourth moment of the Dedekind zeta-function over the Gaussian number field}, { Functiones et Approximatio} {\bf 31} (2003) 23--92.

\bibitem{bomb hej}{E. Bombieri, D. A. Hejhal}, \textit{On the distribution of zeros of linear combinations of Euler products}, {Duke Math. J.} {\bf 80} (1995) 821--862

\bibitem{chandra nara}{K. Chandrasekharan, R. Narasimhan}, \textit{The approximate functional equation for a class of zeta functions}, {Math. Ann.} {\bf 152} (1963) 30--64

\bibitem{con farm}{J. B. Conrey and D. W. Farmer}, \textit{Mean values of $L$-functions and symmetry}, {Internat. Math. Res. Notices }
{\bf 17} (2000)  883--908

\bibitem{cfkrs}{J. B. Conrey, D. W. Farmer, J. P. Keating, M. O. Rubinstein, N. C. Snaith}, \textit{Integral moments of $L$-functions}, {Proc. London Math. Soc.} (3) 91 (2005) 33--104

\bibitem{conrey ghosh}{J. B. Conrey, A. Ghosh}, \textit{A conjecture for the sixth power moment of the Riemann zeta-function}, {Int. Math. Res. Not.} {\bf 15} (1998) 775--780

\bibitem{conrey ghosh 2}{J. B. Conrey, A. Ghosh}, \textit{On the Selberg class of Dirichlet series: small degrees}, {Duke Math. J.} {\bf 72} (1993) 673--693.

\bibitem{conrey gonek}{J. B. Conrey, S. Gonek}, \textit{High moments of the Riemann zeta function}, {Duke Math. J.} {\bf 107} (2001) 577--604.

\bibitem{fomenko}{O.~M.~Fomenko}, \textit{Mean values connected with the Dedekind zeta function}, {J. Math. Sci.} (3) {\bf{150}}  (2008) 2115--2122.

\bibitem{hybrid}{S. Gonek, C. P. Hughes, J. P. Keating}, \textit{A Hybrid Euler-Hadamard Product for the Riemann Zeta Function}, {Duke Math J.} {\bf 136} (3)  (2007) 507--549

\bibitem{me}{W. Heap}, \textit{The twisted second moment of the Dedekind zeta function of a quadratic field}, preprint available at \href{http://arxiv.org/abs/1211.2182}{arxiv: 1211.2182}

\bibitem{iwaniec kowalski}{H. Iwaniec, E. Kowalski}, \textit{Analytic Number Theory}, {American Math. Soc.}, {Vol. 53}, {Colloquium Publications} 2004 

\bibitem{keating snaith}{J. P. Keating, N. C. Snaith}, \textit{Random matrix theory and $\zeta(1/2+it)$}, {Comm. Math. Phys.} {\bf 214} (2000) 57--89

\bibitem{liu ye}{J. Liu, Y. Ye}, \textit{Superposition of zeros of distinct $L$-functions}, {Forum Math.} {\bf 14} (2002) 419--455

\bibitem{mv 0}{H. Montgomery, R. Vaughan}, \textit{Hilbert's inequality}, {J. London Math. Soc.} (2) {\bf 8} (1974) 73--82 

\bibitem{mv}{H. Montgomery, R. Vaughan}, \textit{Multiplicative number theory I: classical theory}, {Cambridge studies in advanced mathematics {\bf 97}}, Cambridge university press 2006

\bibitem{mot}{Y. Motohashi}, \textit{A note on the mean value of the Dedekind zeta function of the quadratic field}, {Math. Ann.} {\bf 188} (1970) 123--127 

\bibitem{muller ded}{W.~M\"uller}, \textit{The Mean Square of the Dedekind Zeta Function in Quadratic Number Fields}, {Math. Proc. Cam. Phil. Soc.} {\bf{106}} (1989) 403--417.

\bibitem{mur}{M. R. Murty}, \textit{Selberg’s conjectures and Artin L-functions}, {Bull. Amer. Math. Soc.} (N.S.) {\bf 31} (1994)
1--14.

\bibitem{neukirch}{J. Neukirch}, \textit{Algebraic Number Theory}, Springer-Verlag Berlin Heidelberg (1999) 

\bibitem{mertens}M. Rosen, \textit{A generalization of Mertens’ theorem}, {J. Ramanujan Math. Soc.} {\bf 14} 
(1999) 1--19

\bibitem{ram}K. Ramachandra, \textit{Application of a theorem of Montgomery and Vaughan to the zeta function}, {J. London Math. Soc.} (2) {\bf 10} (1975) 482--486

\bibitem{rud sarn}{Z. Rudnick, P. Sarnak} \textit{Zeros of principal L-functions and random matrix theory},  {Duke J. of Math.}  {\bf 81} (1996), 269--322

\bibitem{sarnak 0}{P.~Sarnak}, \textit{Fourth Moments of Grossencharakteren Zeta Functions}, {Comms. Pure and App. Math.} {\bf38} (1985) 167--178. 

\bibitem{sarnak 1}{P. Sarnak}, \textit{Quantum chaos, symmetry and zeta functions}, {Curr. Dev. Math.}, 84--115, 1997 


\bibitem{selberg} {A. Selberg}, \textit{Old and new conjectures and results about a class of Dirichlet series}, {Collected papers}, {vol. 2}, {Springer-Verlag}, {Berlin Heidelberg New York}, 1991.
\end{thebibliography}
\end{document}